\documentclass[10pt,a4paper]{article}
\usepackage[latin1]{inputenc}
\usepackage{xcolor}
\usepackage{pgfplots}

\usepackage[english]{babel}
    
\usepackage{graphicx}
\usepackage{mathtools}

\allowdisplaybreaks

\usepackage{xcolor}
\usepackage[colorinlistoftodos,bordercolor=orange,backgroundcolor=orange!20,linecolor=orange,textsize=scriptsize]{todonotes}

\newcommand{\eqdef}{\overset{\text{def}}{=}} 
\newcommand{\EE}[1]{\mathbb{E}\left[#1\right] }

\newcommand{\norm}[1]{\lVert#1\rVert}
\newcommand{\dotprod}[1]{\left< #1\right>}

\usepackage{xcolor}
\usepackage{color}
\usepackage{graphicx}

\usepackage{amsfonts}
\usepackage{mathrsfs, amsmath}
\usepackage{amsthm}
\usepackage{amssymb} 

\usepackage{mdframed} 
\usepackage{thmtools}
\usepackage{multirow}

\definecolor{shadecolor}{gray}{0.95}
\declaretheoremstyle[
headfont=\normalfont\bfseries,
notefont=\mdseries, notebraces={(}{)},
bodyfont=\normalfont,
postheadspace=0.5em,
spaceabove=1pt,
mdframed={
  skipabove=8pt,
  skipbelow=8pt,
  hidealllines=true,
  backgroundcolor={shadecolor},
  innerleftmargin=4pt,
  innerrightmargin=4pt}
]{shaded}

\usepackage[pagebackref=true,backref=true]{hyperref}
\renewcommand*{\backrefalt}[4]{%
    \ifcase #1 \footnotesize{(Not cited.)}%
    \or        \footnotesize{(Cited on page~#2)}%
    \else      \footnotesize{(Cited on pages~#2)}%
    \fi}
    
\hypersetup{
    colorlinks,
    linkcolor={red!50!black},
    citecolor={blue!50!black},
    urlcolor={blue!80!black}
}
 
\usepackage{fancyhdr}
\usepackage{amsmath}
\usepackage{booktabs}
\usepackage{amssymb}
\usepackage{algorithm}
\usepackage{algorithmic}
\usepackage[scale=0.7]{geometry}

\newcommand \RR {\mathbb{R}}

\newcommand{\scp}[2]{\langle #1 \,,\, #2\rangle}

\DeclareMathOperator{\signum}{sign}     
\DeclareMathOperator{\argmin}{argmin}        

\DeclareMathOperator{\dist}{dist}

\numberwithin{equation}{section}

\declaretheorem[style=shaded,within=section]{definition}
\declaretheorem[style=shaded,sibling=definition]{theorem}

\declaretheorem[style=shaded,sibling=definition]{corollary}

\declaretheorem[style=shaded,sibling=definition]{lemma}

\declaretheorem[style=shaded,sibling=definition]{remark}

\usepackage{subfig}

\title{Adaptive Bregman-Kaczmarz: An Approach to Solve Linear Inverse Problems with Independent Noise Exactly}
\author{Lionel Tondji\thanks{Center for Industrial Mathematics, Fachbereich Mathematik/Informatik, University of Bremen, 28334 Bremen and Institute for Analysis and Algebra, TU Braunschweig, 38092 Braunschweig, Germany, \texttt{tondji@uni-bremen.de}. This work has received funding from the European Union's Framework Programme for Research and Innovation Horizon 2020 (2014-2020) under the Marie Sk\l odowska-Curie Grant Agreement No. 861137.},
  Idriss Tondji\thanks{African Institute for Mathematical Sciences (AIMS), AMMI Senegal, \texttt{itondji@aimsammi.org}},
  Dirk Lorenz\thanks{Center for Industrial Mathematics, Fachbereich Mathematik/Informatik, University of Bremen, 28334 Bremen and Institute for Analysis and Algebra, TU Braunschweig, 38092 Braunschweig, Germany, \texttt{d.lorenz@uni-bremen.de}}}
\date{\today}

\begin{document}

\maketitle

\begin{abstract}
  We consider the block Bregman-Kaczmarz method for finite dimensional linear inverse problems.
  The block Bregman-Kaczmarz method uses blocks of the linear system and performs iterative steps with these blocks only.
  We assume a noise model that we call \emph{independent noise}, i.e. each time the method performs a step for some block, one obtains a noisy sample of the respective part of the right-hand side which is contaminated with new noise that is independent of all previous steps of the method.
  One can view these noise models as making a fresh noisy measurement of the respective block each time it is used.
  In this framework, we are able to show that a well-chosen adaptive stepsize of the block Bergman-Kaczmarz method is able to converge to the exact solution of the linear inverse problem.
  The plain form of this adaptive stepsize relies on unknown quantities (like the Bregman distance to the solution), but we show a way how these quantities can be estimated purely from given data.
  We illustrate the finding in numerical experiments and confirm that these heuristic estimates lead to effective stepsizes.
\end{abstract}

\noindent
\textbf{Keywords:} Randomized Bregman-Kaczmarz method, adaptive stepsize, inverse problems
\medskip

\noindent
\textbf{AMS Classification:}
65F10, 
15A29, 
65F20, 

\section{Introduction}
In a finite dimensional linear inverse problem, we have a linear map $\mathbf{A}$ that represents an indirect measurement of some unknown quantity $\hat{x}$. Since the real world measurements always contain noise, one never actually sees $b = \mathbf{A}x$, but always a noisy version of this. The aim is now, to obtain an approximation to $\hat{x}$ only from the knowledge of the noisy data and the linear map $\mathbf{A}$ (see~\cite{engl1996regularization,mueller2021}).
One specific class of iterative methods are row-action or block-action methods that only use single rows or blocks of rows of the linear operator (represented as a matrix) for each iteration.
Probably the oldest such method is the Kaczmarz method~\cite{Kac37} that is also known as the algebraic reconstruction technique~\cite{gordon1970algebraic}.
In this work, we consider a generalization of the Kaczmarz method, namely the Bregman-Kaczmarz method~\cite{lorenz2014linearized,LWSM14,P15} which is a method to solve minimization problems
\begin{equation} 
\label{eq:PB}
\hat x \eqdef \operatorname*{arg min}_{x \in \RR^n} f(x) \quad \text{subject to} \quad  \mathbf{A}x=b,
\end{equation}
for a strongly convex function $f$ which is finite everywhere. Due to strong convexity, this problem has a unique solution if the system $\mathbf{A}x=b$ is consistent.

Our main result in this article is that we show that it is possible to calculate the \emph{exact} solution of~\eqref{eq:PB} even if one never sees the noise-free right-hand side $b$, but each time that one queries a row or block of the matrix, one obtains the respective entry or block of entries of the right-hand side, contaminated by independent noise.

\subsection{Problem statement}
\label{sec:prob-statment}

We give a more formal statement of the problem:
\begin{enumerate}
\item We are given a matrix $\mathbf{A}\in\RR^{m\times n}$ and a strongly convex function $f:\RR^{n}\to\RR$.
\item The matrix is partitioned into $M$ blocks of rows, i.e. 
  \begin{align*}
    \mathbf{A} =
    \begin{pmatrix}
      \mathbf{A}_{(1)}\\\vdots\\\mathbf{A}_{(M)}
    \end{pmatrix}
  \end{align*}
  with $\mathbf{A}_{(i)}\in\RR^{m_{i}\times n}$ where, without loss of generality, we assume that the blocks consist of consecutive rows which can always be achieved by reordering the rows.
\item Assume that there exists ``true data $b\in\RR^{m}$'' such that the system $\mathbf{A}x = b$ is consistent, we aim to compute the minimum-$f$-solution of $\mathbf{A}x = b$, i.e. the unique solution $\hat{x}$ of 
\begin{align*}
\argmin_{x\in\RR^{n}} f(x) \quad \text{subject to}\quad \mathbf{A}x = b.
\end{align*}
\item We are not given $b$, but each time we query the $i$-th block of $b$, i.e. $b_{(i)}\in\RR^{m_{i}}$ we get a perturbation $\tilde{b}_{(i)}\in\RR^{m_{i}}$ which is $\tilde{b}_{(i)} = b_{(i)} + \varepsilon_{(i)}$ where $\varepsilon_{(i)}$ is a random vector with zero mean and variance $\sigma_{i}^{2}$, i.e.
  \begin{align*}
    \EE{\varepsilon_{(i)}} = 0,\quad  \text{and}\quad \EE{\norm{\varepsilon_{(i)}}^{2}} = \sigma_{i}^{2}.
  \end{align*}
\end{enumerate}
The noise model in the last point is different from other noise models where it is often assumed that a single fixed noisy measurement $b^{\delta}$ is given which fulfills a deterministic error condition $\norm{b-b^{\delta}}\leq \delta$. We call our random noise model \emph{independent noise}. Practically this means that each time we consider the $i$-th block of the right-hand side, we have made a new fresh measurement of $\hat{x}$ with the $i$-th block $\mathbf{A}_{(i)}$ of the measurement matrix $\mathbf{A}$. We denote by
\begin{align*}
\sigma \eqdef (\sum_{i=1}^{M}\sigma_{i}^{2})^{1/2}
\end{align*}
the \emph{total noise level}. We also use the notation 
\begin{align*}
\norm{\mathbf{A}}_{\square} \eqdef \left( \sum\limits_{i=1}^{M}\norm{\mathbf{A}_{(i)}}_2^2 \right)^{1/2}
\end{align*}
for the blockwise mixture of spectral and Frobenius norm (indeed for $M=1$ we have $\norm{\mathbf{A}}_{\square} = \norm{\mathbf{A}}_{2}$ and for $M=m$ we have $\norm{\mathbf{A}}_{\square} = \norm{\mathbf{A}}_{F}$). 

The strongly convex function $f$ is used to resolve the problem of non-uniqueness in the case that $n>m$. It is also used to impose prior knowledge on the solution, e.g. one can promote sparsity of the solution by using $f(x) = \lambda\norm{x}_{1} + \norm{x}_{2}^{2}$ for $\lambda>0$~\cite{LWSM14}.

\subsection{The Bregman-Kaczmarz method}
\label{sec:bregman-kaczmarz}

The Bregman-Kaczmarz method for the solution of~\eqref{eq:PB} with a fixed given right-hand side $\tilde{b}$ and sampling of single rows of $\mathbf{A}$ works as follows: Given an initialization $x_{0}^{*}\in\RR^{n}$, compute $x_{0} = \nabla f^{*}(x_{0}^{*})$ (where $f^{*}$ denotes the convex conjugate of $f$, see~\cite{rockafellar1970}) and with a given sequence $\eta_{k}$ of stepsizes choose in each iteration a (random) row-index $i$ and perform the update 
\begin{equation}
\begin{aligned}
\label{eq:noisy_kaczmarz}
    x^{*}_{k+1} &= x^{*}_k - \eta_k \frac{\langle a_{i}, x_k \rangle - \Tilde{b}_{i}}{\|a_{i}\|^2_2} \cdot  a_{i}  \\
    x_{k+1} &= \nabla f^*(x^*_{k+1}),
\end{aligned}
\end{equation}
The method is a special case of the (relaxed) randomized Bregman projections~\cite{lorenz2014linearized}.
In this paper, we consider a block version where one chooses a random \emph{block} $\mathbf{A}_{(i)}$ from the given partition of $\mathbf{A}$. The choice of a block is random in each iteration and for simplicity we assume that the probability $p_{i}$ to choose the $i$-th block is proportional to the square of the spectral norm of the block matrix, i.e. $p_{i} = \norm{\mathbf{A}_{(i)}}^{2}/\norm{\mathbf{A}}_{\square}^{2}$. The resulting method is detailed in Algorithm~\ref{alg:BK}. Recall that $\tilde{b}_{(i)}$ is a noisy version of $b_{(i)}$ according to the model of independent noise as described in the problem statement in Section~\ref{sec:prob-statment}.

\begin{algorithm}[ht]
  \caption{Adaptive Bregman-Kaczmarz method (aBK)}
  \label{alg:BK}
  \begin{algorithmic}[1]
    \STATE{\textbf{Given:} matrix $\mathbf{A}\in\RR^{m\times n}$ and a ``sampler'' for blocks $b_{(i)}$ of the right hand side}
    \STATE {\textbf{Input:} Choose  $x_0^* \in \RR^n$, set $x_0 = \nabla f^*(x^*_{0})$ and choose stepsizes $\eta_{k}$.}
    \STATE{\textbf{Output:} (approximate) solution of $\min\limits_{x\in \RR^n} f(x) \quad \text{s.t.} \quad  \mathbf{A}x=b$.}
    \STATE initialize $k=0$
    \REPEAT
    \STATE choose a block index $i_k=i\in \left\{ 1,\dots, M \right\}$ with probability $p_i = \frac{\|\mathbf{A}_{(i)}\|^2_2}{\norm{\mathbf{A}}_{\square}^2}$.
    \STATE get a perturbation of the $i$-th block of $b$, $\tilde{b}_{(i)} = b_{(i)} + \varepsilon_{(i)}$ with $ \EE{\varepsilon_{(i)}} = 0, \EE{\norm{\varepsilon_{(i)}}^{2}} = \sigma_{i}^{2}.$
    \STATE update $x^{*}_{k+1} = x^{*}_k - \eta_k \frac{\mathbf{A}_{(i)}^{\top} (\mathbf{A}_{(i)}x_k - \Tilde{b}_{(i)})}{\|\mathbf{A}_{(i)}\|^2_2}$ 
    \vspace{0.2cm}
    \STATE update $x_{k+1} = \nabla f^*(x^*_{k+1})$
    \vspace{0.2cm}
    \STATE increment $k =  k+1$
    \UNTIL{a stopping criterion is satisfied}
    \RETURN $x_{k+1}$
  \end{algorithmic}
\end{algorithm}

\subsection{Related work}
\paragraph{Randomized Kaczmarz.}
 In a large data regime, full matrix operations can be expensive or even infeasible. Then it appears desirable to use iterative algorithms
with low computational cost and storage per iteration that produce good approximate solutions of~\eqref{eq:PB} after relatively few iterations. The Kaczmarz method~\cite{Kac37} and its randomized variants~\cite{SV09,gower2015randomized} are used to compute the minimum $\ell_2$-norm solutions of consistent linear systems.  In each iteration $k$, a row vector $a_i^\top$
of $\mathbf{A}$ is chosen at random from the system $\mathbf{A}x = b$ and the current iterate $x_k$ is projected onto the solution space of that equation to obtain $x_{k+1}$. It has been observed that the convergence of the randomized Kaczmarz (RK) method can be accelerated by introducing relaxation. In a relaxed variant of RK, a step is taken in the direction of this projection with the size of the step depending on a relaxation parameter. Explicitly, the relaxed RK update is given by
\begin{equation}
    \label{eq:RK}
    x_{k+1} = x_k - \eta_{k} \frac{\langle a_i, x_k \rangle - b_i}{\|a_i\|^2_2} \cdot a_i
\end{equation}
with initial values $x_0 = 0$, where the $\eta_{k} \in (0,2)$ is the stepsize. This update rule requires low cost per iteration and storage of order $\mathcal{O}(n)$. The first convergence rate for RK was given in~\cite{SV09} for consistent system and $\eta_k=1$. 

To further improve the efficiency of the RK method, the idea of parallelism was used, such as in ~\cite{gower2019adaptive,moorman2021randomized,necoara2019faster}. The basic idea is to use multiple rows
in each step of the iteration. This will increase the cost of each step of the iteration, but it can reduce the number of iterations as expected. In addition to the above work, there are many studies aimed at accelerating the (randomized) Kaczmarz method by the use of block strategies, sampling schemes, or extrapolation; for example, see ~\cite{bai2018greedy,eldar2011acceleration,haddock2021greed,liu2016accelerated,miao2022greedy,needell2014paved,steinerberger2021weighted}. 

Needell, in~\cite{needell2010randomized} extended the result of~\cite{SV09} to the inconsistent case when only a noisy right-hand side $\Tilde{b}$ is given where $\Tilde{b} = b + \varepsilon$ with arbitrary noise vector $\varepsilon$ such that $|\varepsilon_i| \leq \sigma \|a_i\|$. They proved that in the case $\eta_k = 1$, the iterates in Eq~\eqref{eq:RK} are expected to reach an error threshold in the order of the noise-level with the same rate as in the noiseless case, cf~\cite{needell2010randomized,SV09}. An approach was just recently proposed by Haddock et al.~\cite{haddock2022quantile} to recover the true solution, with a high likelihood as long as the right-hand side has sufficiently small corrupted entries, for a certain class of random matrices. They considered the setting where a fraction of the entries have been corrupted (possibly by arbitrarily large numbers). Their result was generalized for any matrix in~\cite{steinerberger2023quantile} and for sparse solutions in~\cite{zhang2022quantile}.

Just recently, Marshall et al. in~\cite{marshall2023optimal} proposed an approach that is complementary to the ones in~\cite{haddock2022quantile,steinerberger2023quantile,zhang2022quantile}. Instead of considering the setting where a small fraction of the right-hand side entries have been corrupted, they allow for the corruption of all elements of $b$ but assume that corruptions are independent zero-mean random variables so that the noisy right-hand side $\Tilde{b}$ is fixed. They showed that we can recover the solution $\hat x$ to any accuracy if we have access to a sufficient number of equations when the rows are sampled without replacement. This assumption guarantees that the noise in each iterate is independent from previous noise the algorithm has seen until then (and we modify this assumption to the assumption of independent noise). However, their results hold for only one pass over the full matrix, i.e. the number of rows of the matrix. Therefore with this restriction is not possible in practice to recover the true solution $\hat x$ after one epoch only.

In the context of inverse problems, Kaczmarz-type methods are often used for non-linear problems \cite{haltmeier2007kaczmarz,kaltenbacher2008} but also linear inverse problems have been considered, e.g. in~\cite{elfving2014semi,jiao2017preasymptotic} and in Hilbert space in~\cite{lu2022stochastic}.

\paragraph{Randomized sparse Kaczmarz.}
The randomized Kaczmarz method has been adapted to the randomized sparse Kaczmarz method (RSK)~\cite{LWSM14,LS19} which has almost the same low cost and storage requirements and has shown good performance in approximating sparse solutions of large consistent linear systems. It belongs to a more general family of methods called Bregman-Kaczmarz (BK) methods. The Bregman-Kaczmarz~\cite{10178390} uses two variables $x^*_{k}$
 and $x_k$ and its update is given by :
 \begin{equation}
\begin{aligned}
\label{eq:Bregman_kaczmarz}
    x^{*}_{k+1} &= x^{*}_k - \eta_k \frac{\langle a_{i}, x_k \rangle - b_{i}}{\|a_{i}\|^2_2} \cdot  a_{i}  \\
    x_{k+1} &= \nabla f^*(x^{*}_{k+1})
\end{aligned}
\end{equation}
with initial values $x_0 = \nabla f^*(x^*_0)$. The papers~\cite{lorenz2014linearized,LS19} analyze this method by interpreting it as a sequential, randomized Bregman projection method (where the Bregman projection is done with respect to the function $f$).
Variations of RSK including block/averaging variants~\cite{du2020randomized,P15}, averaging methods~\cite{tondji2022faster} and adaptions to least squares problems are given in \cite{schopfer2022extended,zouzias2013randomized}. In those variants, one usually needs to have access to more than one row of the matrix $\mathbf{A}$ at the cost of increasing the memory.

It has been shown in~\cite{LS19} that for a consistent system $\mathbf{A}x=b$
 with an arbitrary matrix $\mathbf{A}$ the iterates $x_k$ of the sparse Kaczmarz $(\eta_k = 1)$ converge linearly to the unique solution $\hat x$  of the regularized Basis Pursuit problem, which is Eq~\eqref{eq:Bregman_kaczmarz} with $f(x):= \lambda \cdot \|x\|_1 + \tfrac{1}{2}\|x\|_2^2$. More precisely under the assumption that the row index $i_k$ at iteration $k$ is chosen
randomly with probability proportional to $\|a_{i_k}\|^2$ they proved that :
\begin{equation}
\label{eq:linear_conv}
    \mathbb{E}[\|x_k -\hat{x}\|^2_2] \leq 2(1 - q)^k f(\hat x)
\end{equation}
where $q \in (0,1)$. For more details, we refer the reader to~\cite{LS19} and an acceleration for update~\eqref{eq:Bregman_kaczmarz} was given in~\cite{10178390} for any objective function $f$. In~\cite{LS19}, they also consider the case of a noisy linear system: instead of having access to the right-hand side of
the consistent linear system $\mathbf{A}x = b$, we are given $\Tilde{b} = b + \varepsilon$, where the entries of $\varepsilon$ satisfy $|\varepsilon_i| \leq \sigma \|a_i\|$. Under these assumptions, they proved that the sparse Kaczmarz with $\eta_k = 1$ satisfies
\begin{equation}
\label{eq:linear_conv_noisy}
    \mathbb{E}[\|x_k -\hat{x}\|^2_2] \leq 2(1 - q)^k f(\hat x) + \frac{\sigma^2}{q}
\end{equation}
that is, we converge in expectation until we reach some ball of radius $\frac{\sigma}{\sqrt{q}}$ around the solution. 

\subsection{Contribution}
In this paper, we make the following contributions:
\begin{itemize}
    \item We show that an adaptive stepsize similar to the one proposed in~\cite{marshall2023optimal} can also be applied in the case of the Bregman-Kaczmarz setting.
    \item  We extend the analysis to the block case.
    \item We replace the assumption that the rows or blocks are sampled without replacement with the assumption of independent noise and prove that the method with adaptive stepsize does indeed converge to the true solution if the data obey the independent noise assumption.
    \item We illustrate that this adaptive stepsize can be implemented in practice by using a heuristic similar to the one from~\cite{marshall2023optimal}, that estimates the needed hyperparameter only from given data and shows that this heuristic works well in practice.
    \item We investigate the convergence of the method with adaptive stepsize in detail and show that it leads to linear convergence of the iterates at the beginning and an $\mathcal{O}(1/k)$ rate of convergence later on.
\end{itemize}

\subsection{Outline}
The remainder of the paper is organized as follows. Section~\ref{sec:basicnotions} provides notations and a brief overview on convexity and Bregman distances. We state error bound conditions that hold for our objective function. Section~\ref{sec:convergence_analysis} provides convergence guarantees for our proposed method which depend on some hyperparameters. In Section~\ref{sec:heuristic}, we derive heuristics estimation of those hyperparameters. In Section~\ref{sec:numerical_experiment}, numerical experiments demonstrate the effectiveness of our method and provide insight regarding its behavior and its hyperparameters. Finally,
Section~\ref{sec:conclusion} draws some conclusions.

\section{Notation and basic notions}
\label{sec:basicnotions}

We will analyze the convergence of the Algorithm~\ref{alg:BK} with the help of the Bregman distance with respect to the objective function $f$. To this end
we recall some well-known concepts and properties of convex functions.

Let $f:\RR^n \to \RR$ be convex and since it is finite everywhere, it is also continuous. The \emph{subdifferential} of $f$ at any $x \in \RR^n$ is defined by
\[
\partial f(x) \eqdef \{x^* \in \RR^n| f(y) \ge f(x) + \langle x^*, y-x \rangle, \forall\, y \in \RR^n \},
\]
which is nonempty, compact and convex. The function  $f$ is said to be \emph{$\alpha$-strongly convex}, if for all $x,y \in \RR^n$ and subgradients $x^* \in \partial f(x)$ we have
\[
f(y) \geq f(x) + \langle x^*, y-x \rangle + \tfrac{\alpha}{2} \cdot \|y-x\|_2^2 \,.
\]
If $f$ is $\alpha$-strongly convex, then $f$ is coercive, i.e. $\lim_{\|x\|_2 \to \infty} f(x)=\infty$,
and its \emph{Fenchel conjugate} $f^{*}:\RR^n \to \RR$ given by 
\[
f^*(x^*)\eqdef \sup_{y \in \RR^n} \scp{x^*}{y} - f(y)
\]
is also convex, finite everywhere and coercive. Additionally, $f^*$ is differentiable with a \emph{Lipschitz-continuous gradient} with constant $L_{f^*}=\frac{1}{\alpha}$, i.e. for all $x^*,y^* \in \RR^n$ we have
\[
\|\nabla f^*(x^*)-\nabla f^*(y^*)\|_2 \le L_{f^*} \cdot \|x^*-y^*\|_2 \,,
\]
which implies the estimate
\begin{align}\label{eq:Lip}
f^*(y^*)
\le f^*(x^*) +\scp{\nabla f^*(x^*)}{y^*-x^*} + \tfrac{L_{f^*}}{2}\|x^{*}-y^{*}\|_2^2. 
\end{align}

\begin{definition} \label{def:D}
The \emph{Bregman distance} $D_f^{x^*}(x,y)$ between $x,y \in \RR^n$ with respect to $f$ and a subgradient $x^* \in \partial f(x)$ is defined as
\[
D_f^{x^*}(x,y) \eqdef f(y)-f(x) -\scp{x^*}{y - x}\,.
\]
\end{definition}
Fenchel's equality states that $f(x) + f^*(x^*) = \scp{x}{x^*}$ if $x^*\in\partial f(x)$ and implies that the Bregman distance can be written as
\[
D_f^{x^*}(x,y) = f^*(x^*)-\scp{x^*}{y} + f(y)\,.
\]

We will use the function
\begin{equation} \label{eq:spf}
f(x) \eqdef \lambda\|x\|_1 + \tfrac{1}{2}\|x\|_2^{2}
\end{equation}
which is strongly convex with constant $\alpha=1$. Its conjugate function is the soft shrinkage operator $S_{\lambda}$ which is defined componentwise by $(S_{\lambda}(x))_j = \max\{|x_j|-\lambda,0\} \cdot \signum(x_j)$
\[ 
f^{*}(x^{*}) = \tfrac{1}{2}\|S_{\lambda}(x^{*})\|_2^{2}, \quad \mbox{with} \quad \nabla f^{*}(x^{*}) = S_{\lambda}(x^{*}) \,.
\]
For any $x^*=x+\lambda \cdot s \in \partial f(x)$, with  $s \in \partial (\|x\|_1)$, we have
\[
D_f^{x^*}(x,y)=\tfrac{1}{2}\|x-y\|_2^2 + \lambda \cdot(\|y\|_1-\scp{s}{y}) \,.
\]
which gives us $D_f^{x^*}(x,y)=\tfrac{1}{2}\|x-y\|_2^2$ for $\lambda = 0$.

In general we have for an $\alpha$-strongly convex function that 
\begin{align} 
\label{eq:D}
  D_f^{x^*}(x,y) \geq \frac{\alpha}{2} \|x-y\|_2^2, 
\end{align}
see~\cite{lorenz2014linearized}, and hence $D_f^{x^*}(x,y) = 0 \Leftrightarrow x=y.$


To obtain convergence rates for the solution algorithm, we will estimate the Bregman distance of the iterates to the solution $\hat x$ by error bounds of the form $D_f^{x^*}(x,\hat x) \leq \theta^{-1} \cdot \|\mathbf{A}x - b\|^2_2$. We will see that such error bounds always hold if $f$ has a Lipschitz-continuous gradient. But they also hold under weaker conditions. For example problem~\eqref{eq:PB} with objective function defined by~\eqref{eq:spf}, it holds that 
\begin{equation*}
    D_f^{x^*}(x,\hat x) \leq \theta(\hat x)^{-1} \cdot \|\mathbf{A}x - b\|^2_2,
\end{equation*}
we refer the reader to~\cite[Theorem 3.9]{schopfer2022extended} for more details on the constant $\theta(\hat x)$. 
To clarify the assumptions under which such error bounds hold for more general objective functions, we are going to define notions~\cite{schopfer2022extended} such as calmness, linearly regularity, and linear growth. Let $B_2$ denote the closed unit ball of the $\ell_2$-norm.

\begin{definition}
The (set-valued) subdifferential mapping $\partial f:\RR^n \rightrightarrows \RR^n$ is \emph{calm} at $\hat{x}$ if there are constants $\varepsilon,\ell>0$ such that
\begin{equation}
\partial f(x) \subset \partial f(\hat{x}) + \ell \cdot \|x-\hat{x}\|_2 \cdot B_2 \quad \mbox{for any $x$ with} \quad \|x-\hat{x}\|_2\le \varepsilon\,. \label{eq:calm}
\end{equation}
\end{definition}

Calmness is a local growth condition akin to Lipschitz-continuity
of a gradient mapping, but for fixed $\hat x$. Of course, any Lipschitz-continuous
gradient mapping is calm everywhere. Moreover, the subdifferential mapping of any convex piecewise linear quadratic function is calm everywhere. In particular, this holds for the function defined in eq~\eqref{eq:spf}. For more examples on calmness, we refer the reader to~\cite{schopfer2022extended}.

\begin{definition}
Let $\partial f(x) \cap \mathcal{R}(\mathbf{A}^{\top}) \neq \emptyset.$ Then the collection $\{\partial f(\hat x), \mathcal{R}(\mathbf{A}^{\top})\}$ is linearly regular, if there is a constant $\zeta >0$ such that for all $x^{*} \in \RR^n$ we have
\begin{equation}
\label{eq:linreg}
    \dist(x^{*}, \partial f(\hat x) \cap \mathcal{R}(\mathbf{A}^{\top})) \leq \zeta \cdot \bigg(\dist(x^{*}, \partial f(\hat x)) + \dist(x^{*}, \mathcal{R}(\mathbf{A}^{\top}))\bigg)
\end{equation}
\end{definition}
The collection $\{\partial f(\hat x), \mathcal{R}(\mathbf{A}^{\top})\}$ is linearly regular, if $\partial f(\hat x)$ is polyhedral (which holds for piecewise linear-quadratic $f$ in particular).

\begin{definition}
We say the subdifferential mapping of $f$ \emph{grows at most linearly}, if there exist $\rho_1,\rho_2 \ge 0$ such that for all $x \in \RR^n$ and $x^* \in \partial f(x)$ we have
\begin{equation} \label{eq:lineargrowth}
\|x^*\|_2 \le \rho_1 \cdot \|x\|_2 + \rho_2\,.
\end{equation}
\end{definition}
Any Lipschitz-continuous gradient mapping grows at most linearly. Furthermore, the subdifferential mappings of the functions from~\eqref{eq:spf} grow at most linearly.

\begin{theorem}(cf.~\cite[Theorem 3.9]{schopfer2022extended})
\label{th:error_bounds_equality}
Consider the linearly constrained optimization problem~\eqref{eq:PB} and strongly convex $f:\RR^n \to \RR$.  If its subdifferential mapping grows at most linearly, is calm at the unique solution $\hat x$ of \eqref{eq:PB}, and if the collection $\{\partial f(\hat x), \mathcal{R}(\mathbf{A}^{\top})\}$ is linearly regular, then there exists $\theta(\hat x)$ such that for all $x \in \RR^n$ and $x^* \in \partial f(x) \cap \mathcal{R}(\mathbf{A}^{\top})$ we have the global error bound :
\begin{equation}
    \label{eq:EB}
    D_f^{x^*}(x,\hat x)  \leq \dfrac{1}{\theta(\hat x)}  \cdot \|\mathbf{A}x-b\|_2^2
\end{equation}
\end{theorem}

\section{Convergence analysis}
\label{sec:convergence_analysis}

To start the convergence analysis of the Bregman-Kaczmarz method (Algorithm~\ref{alg:BK}), we first characterize how the Bregman distance of the iterates to the solution $\hat x$ decays from one iteration to the next one. By proper scaling of $f$, we can always assume that $f$ is $1$-strongly convex, i.e. we have $\alpha=1$, and consequently by~\eqref{eq:D} we can assume that the Bregman distance is an upper bound to the squared norm, i.e. we always have 
\begin{align*}
\frac{1}{2}\norm{x_{k}-\hat{x}}^{2} \leq D_{f}^{x_{k}^{*}}(x_{k},\hat{x}) .
\end{align*} 

\begin{lemma}
\label{lem:nWRSK_x_descent}
Let $(x_k, x_k^*)$ be the two sequences generated by Algorithm~\ref{alg:BK}. Then, it holds that
\begin{align}
\label{eq:descent_lem}
     D_f^{x_{k+1}^*}(x_{k+1},\hat x) &\leq D_f^{x_{k}^*}(x_{k},\hat x) - \frac{\eta_k (2-\eta_k)}{2} \frac{\|\mathbf{A}_{(i_k)}x_k - b_{(i_k)}\|^2}{\|\mathbf{A}_{(i_k)}\|_2^2} \nonumber \\ 
     &+ \frac{\eta_k^2}{2} \frac{\|\varepsilon_{(i_k)}\|^2}{\|\mathbf{A}_{(i_k)}\|^2_2} + \eta_k(1-\eta_k) \varepsilon_{(i_k)}^{\top} \cdot \frac{( \mathbf{A}_{(i_k)}x_k - b_{(i_k)} )}{\|\mathbf{A}_{(i_k)}\|^2_2}
\end{align}
\end{lemma}
 \begin{proof}
 Using  $\tilde{b}_{(i)} = b_{(i)} + \varepsilon_{(i)}$ and defining
 $$\Tilde{x}_k := \hat x + \frac{\mathbf{A}_{(i_k)}^{\top}(\Tilde{b}_{(i_k)} - b_{(i_k)})}{\|\mathbf{A}_{(i_k)}\|^2_2}.$$ Since $\nabla f^*$ is $1$-Lipschitz-continuous, the descent lemma~\eqref{eq:Lip} implies
 \begin{align*}
 D_f^{x_{k+1}^*}(x_{k+1},\Tilde{x}_k) &= f^*(x_{k+1}^*) + f(\Tilde{x}_k) - \langle x_{k+1}^*,\Tilde{x}_k\rangle \\
 &\leq f^*(x_k^*) + \langle \nabla f^*(x_k^*),x_{k+1}^*-x_k^*\rangle + \frac{1}{2}\|x_{k+1}^*-x_k^*\|_2^2 + f(\Tilde{x}_k) - \langle x_{k+1}^*,\Tilde{x}_k\rangle.  
 \end{align*}

 \noindent Using the fact that $x_k=\nabla f^*(x_k^*)$ we conclude
 \begin{align*}
 D_f^{x_{k+1}^*}(x_{k+1},\Tilde{x}_k) &\leq 
 D_f^{x_k^*}(x_k,\Tilde{x}_k) + \dotprod{ x_k, \ - \eta_k \frac{\mathbf{A}_{(i_k)}^{\top} (\mathbf{A}_{(i_k)}x_k - \Tilde{b}_{(i_k)})}{\|\mathbf{A}_{(i_k)}\|^2_2} } \\
 & \hspace{0.5cm} + \frac{1}{2} \Bigl\| - \eta_k \frac{\mathbf{A}_{(i_k)}^{\top} (\mathbf{A}_{(i_k)}x_k - \Tilde{b}_{(i_k)})}{\|\mathbf{A}_{(i_k)}\|^2_2} \Bigr\|_2^2   +  \Big\langle  \eta_k \frac{\mathbf{A}_{(i_k)}^{\top} (\mathbf{A}_{(i_k)}x_k - \Tilde{b}_{(i_k)})}{\|\mathbf{A}_{(i_k)}\|^2_2},\Tilde{x}_k \Big\rangle  \\
 &= D_f^{x_k^*}(x_k,\Tilde{x}_k) - \frac{\eta_k}{\|\mathbf{A}_{(i_k)}\|^2_2} \Big \langle \mathbf{A}_{(i_k)}^{\top}(\mathbf{A}_{(i_k)}x_k - \Tilde{b}_{(i_k)}), x_k - \Tilde{x}_k  \Big \rangle \\
 & \hspace{0.5cm} + \frac{\eta_k^2}{2\|\mathbf{A}_{(i_k)}\|^2_2} \|\mathbf{A}_{(i_k)}x_k - \Tilde{b}_{(i_k)} \|^2_2 \\
 &= D_f^{x_k^*}(x_k,\Tilde{x}_k) + \langle x_{k+1}^* - x_{k}^*, x_k - \Tilde{x}_k\rangle + \frac{\eta_k^2}{2\|\mathbf{A}_{(i_k)}\|^2_2} \|\mathbf{A}_{(i_k)}x_k - \Tilde{b}_{(i_k)} \|^2_2
 \end{align*}
 Unfolding the expression of $ D_f^{x_{k+1}^*}(x_{k+1},\Tilde{x}_k)$ and $D_f^{x_{k}^*}(x_{k},\Tilde{x}_k),$ and adding $f(\hat x)$ to both sides, we obtain the inequality
 \begin{equation}
 \label{eq:nwrsk1}
      D_f^{x_{k+1}^*}(x_{k+1},\hat x) \leq D_f^{x_{k}^*}(x_{k},\hat x) + \langle x_{k+1}^* - x_{k}^*, \Tilde{x}_k - \hat x\rangle + \langle x_{k+1}^* - x_{k}^*, x_k - \Tilde{x}_k\rangle + \frac{\eta_k^2}{2\|\mathbf{A}_{(i_k)}\|^2_2} \|\mathbf{A}_{(i_k)}x_k - \Tilde{b}_{(i_k)} \|^2_2.
 \end{equation}
 By observing that $\Big \langle \mathbf{A}_{(i_k)}^{\top}(\mathbf{A}_{(i_k)}x_k - \Tilde{b}_{(i_k)}), x_k - \hat x  \Big \rangle = \|\mathbf{A}_{(i_k)}x_k - b_{(i_k)}\|^2_2 - \varepsilon_{(i_k)}^{\top} (\mathbf{A}_{(i_k)}x_k - b_{(i_k)})$, we conclude that
 \begin{equation}
 \label{eq:nwrsk2}
      \langle x_{k+1}^* - x_{k}^*, \Tilde{x}_k - \hat x\rangle + \langle x_{k+1}^* - x_{k}^*, x_k - \Tilde{x}_k\rangle =  - \frac{\eta_k}{\|\mathbf{A}_{(i_k)}\|^2_2}\Big \langle \mathbf{A}_{(i_k)}^{\top}(\mathbf{A}_{(i_k)}x_k - \Tilde{b}_{(i_k)}), x_k - \hat x  \Big \rangle,
 \end{equation}
 by rewriting 
 \begin{equation}
 \label{eq:nwrsk3}
\| \mathbf{A}_{(i_k)}x_k - \Tilde{b}_{(i_k)}\|^2_2 = \|\mathbf{A}_{(i_k)}x_k - b_{(i_k)}\|^2_2 + \|\varepsilon_{(i_k)}\|^2_2 - 2 \varepsilon_{(i_k)}^{\top} 
 (\mathbf{A}_{(i_k)}x_k - b_{(i_k)}).
 \end{equation}
 Inserting (\ref{eq:nwrsk2}) and (\ref{eq:nwrsk3}) into (\ref{eq:nwrsk1}), we obtain
 \begin{align*}
      D_f^{x_{k+1}^*}(x_{k+1},\hat x) &\leq D_f^{x_{k}^*}(x_{k},\hat x) - \frac{\eta_k (2-\eta_k)}{2} \frac{\|\mathbf{A}_{(i_k)}x_k - b_{(i_k)}\|^2_2}{\|\mathbf{A}_{(i_k)}\|_2^2} \\ 
      &+ \frac{\eta_k^2}{2} \frac{\|\varepsilon_{(i_k)}\|^2}{\|\mathbf{A}_{(i_k)}\|^2_2} + \eta_k(1-\eta_k) \varepsilon_{(i_k)}^{\top} \cdot \frac{(\mathbf{A}_{(i_k)}x_k - b_{(i_k)})}{\|\mathbf{A}_{(i_k)}\|^2_2}.
 \end{align*}
 \end{proof}

The term in~\eqref{eq:descent_lem} from Lemma~\ref{lem:nWRSK_x_descent} which is
linear in $\varepsilon_{(i_k)}$ is
\begin{align*}
  \eta_{k}(1-\eta_{k})\varepsilon_{(i_{k})}^{\top}\cdot\frac{( \mathbf{A}_{(i_k)}x_k - b_{(i_k)} )}{\|\mathbf{A}_{(i_k)}\|^2_2}
\end{align*}
and vanishes when the stepsize $\eta_k$ is one (i.e. we do full steps) regardless of the structure of the noise, but does not vanish in the case where we use adaptive stepsizes $\eta_k \neq 1$. However, we assume independent noise, i.e. $\varepsilon_{(i_{k})}$ is independent of all previous randomness in the algorithm. Since we also assume that the noise has zero mean, this term will vanish, when we take the expectation over the noise realization.

\begin{lemma}\label{lem:descent-D}
  Assume independent noise (cf. Section~\ref{sec:prob-statment}) and that $(x_{k},x_{k}^{*})$ are generated by Algorithm~\ref{alg:BK} and denote $\gamma = \tfrac{\theta(\hat x)}{\norm{\mathbf{A}}_{\square}^{2}}$ (where $\theta(\hat x)$ comes from Theorem~\ref{th:error_bounds_equality}). Then it holds that
  \begin{align}
    \label{eq:des_l}
    \mathbb{E}_{i_{k}}\mathbb{E}_{\varepsilon}[D_f^{x_{k+1}^*}(x_{k+1},\hat x)] &\leq \bigg( 1 -  \frac{\eta_k (2-\eta_k)}{2} \gamma\bigg) D_f^{x_{k}^*}(x_{k},\hat x) + \frac{\eta_k^2}{2}\frac{\sigma^{2}}{\norm{\mathbf{A}}_{\square}^{2}}. 
  \end{align}
\end{lemma}
\begin{proof}
 In~\eqref{eq:descent_lem} in Lemma~\ref{lem:nWRSK_x_descent} we take the expectation with respect to the choice of $i_{k}$ and the noise realization $\varepsilon$ and get (since the noise realization and the choice of the index are independent)
\begin{align*}
     \mathbb{E}_{i_{k}}\mathbb{E}_{\varepsilon}[D_f^{x_{k+1}^*}(x_{k+1},\hat x) ]\leq D_f^{x_{k}^*}(x_{k},\hat x) - \frac{\eta_k (2-\eta_k)}{2} \mathbb{E}_{i_{k}}\left[ \frac{\|\mathbf{A}_{(i_k)}x_k - b_{(i_k)}\|^2}{\|\mathbf{A}_{(i_k)}\|_2^2}\right] \\ 
     + \frac{\eta_k^2}{2} \mathbb{E}_{i_{k}}\mathbb{E}_{\varepsilon}\left[\frac{\|\varepsilon_{(i_k)}\|^2}{\|\mathbf{A}_{(i_k)}\|^2_2}\right] + \eta_k(1-\eta_k) \mathbb{E}_{\varepsilon}\left[\varepsilon_{(i_k)}^{\top}\right] \cdot \mathbb{E}_{i_{k}}\left[\frac{( \mathbf{A}_{(i_k)}x_k - b_{(i_k)} )}{\|\mathbf{A}_{(i_k)}\|^{2}}\right]
\end{align*}
We have $\mathbb{E}_{\varepsilon}\left[\varepsilon_{(i_k)}^{\top}\right] = 0$ and 
\begin{align*}
  \mathbb{E}_{i_{k}}\mathbb{E}_{\varepsilon}\left[\frac{\|\varepsilon_{(i_k)}\|^2}{\|\mathbf{A}_{(i_k)}\|^2_2}\right] & = \mathbb{E}_{i_{k}}\left[\tfrac{1}{\norm{\mathbf{A}_{(i_{k})}}^{2}}\mathbb{E}_{\varepsilon}\norm{\varepsilon_{(i_{k})}}^{2}\right] = \mathbb{E}_{i_{k}}\left[\frac{\sigma_{i_{k}}^{2}}{\norm{\mathbf{A}_{(i_{k})}}^{2}}\right]\\
                                                                                                  & = \sum_{j=1}^{M}p_{j}\frac{\sigma_{j}^{2}}{\norm{\mathbf{A}_{(j)}}^{2}} = \frac{\sum_{j=1}^M\sigma_{j}^{2}}{\sum_{j=1}^M\norm{\mathbf{A}_{(j)}}^{2}} = \frac{\sigma^{2}}{\norm{\mathbf{A}}_{\square}^{2}}.
\end{align*}
By Theorem~\ref{th:error_bounds_equality} we get
\begin{align*}
  \mathbb{E}_{i_{k}}\left[ \frac{\|\mathbf{A}_{(i_k)}x_k - b_{(i_k)}\|^2}{\|\mathbf{A}_{(i_k)}\|_2^2}\right] & = \sum\limits_{i=1}^M \frac{\|\mathbf{A}_{(i)}\|^2}{\norm{\mathbf{A}}_{\square}^{2}} \frac{\|\mathbf{A}_{(i)}x_k - b_{(i)}\|^2}{\|\mathbf{A}_{(i)}\|_2^2}\\
  & = \tfrac1{\norm{\mathbf{A}}_{\square}^{2}}\norm{\mathbf{A}x_{k}-b}^{2} \geq \tfrac{\theta(\hat x)}{\norm{\mathbf{A}}_{\square}^{2}}D_{f}^{x_{k}^{*}}(x_k,\hat x).
\end{align*}
Combining these inequalities proves the claim.
\end{proof}
Using the rule of total expectation we obtain 
\begin{align}
  \label{eq:des_l_EE}
    \EE{D_f^{x_{k+1}^*}(x_{k+1},\hat x)} &\leq \bigg( 1 -  \frac{\eta_k (2-\eta_k)}{2} \gamma\bigg) \EE{D_f^{x_{k}^*}(x_{k},\hat x)} + \frac{\eta_k^2}{2}\frac{\sigma^{2}}{\norm{\mathbf{A}}_{\square}^{2}}
\end{align}
where the expectation is taken over all randomness involved (i.e. all randomly chosen indices so far and all noise realizations).

\noindent To make the expectation on the left-hand side of~\eqref{eq:des_l} as small as possible, we aim to minimize the right-hand side and the only thing we have available to do so, is the stepsize $\eta_{k}$.
Although the right-hand side is quadratic in $\eta_{k}$, it is not clear how to do that, since we do not have all information that we need available, especially we do not know the distance $D_f^{x_{k}^*}(x_{k},\hat x)$ and the constant $\gamma$.
If we had no noise, i.e. $\sigma^{2}=0$, the right-hand side would be minimal for $\eta_{k}=1$, i.e. we would always do full steps.
We proceed by assuming that the quantities which are needed to achieve this goal are known (even if they are not in practice) and later in Section~\ref{sec:heuristic} show how these quantities can be estimated in practice. The derivation follows along the lines presented in~\cite{marshall2023optimal}.

We unroll the estimate~\eqref{eq:des_l_EE} recursively to get
\begin{align}\label{eq:des-l-EE-recursive}
 \EE{D_{f}^{x_{k+1}^{*}}(x_{k+1},\hat{x})}\leq \prod\limits_{j=0}^{k}\left( 1 - \tfrac{\eta_{j}(2-\eta_{j})}{2} \gamma\right) \EE{D_{f}^{x_{0}^{*}}(x_{0},\hat{x})} + \frac{\sigma^{2}}{\norm{\mathbf{A}}_{\square}^{2}} \sum\limits_{j=0}^{k-1}\frac{\eta_{j}^{2}}{2}\prod\limits_{i=j+1}^{k}\left( 1 - \tfrac{\eta_{i}(2-\eta_{i})}{2}\gamma \right)
\end{align}
(with the convention that the empty product equals one).

Let $\eta$ be the vector of all the values $\eta_{j}$ and define
\begin{align*}
    D_k(\eta) \eqdef \prod_{j=0}^{k} \bigg( 1 -  \frac{\eta_{j} (2-\eta_{j})}{2} \gamma\bigg) \frac{\norm{\mathbf{A}}_{\square}^{2}\,\mathbb{E}[D_f^{x_{0}^*}(x_{0},\hat x)]}{\sigma^2} + \sum_{j=0}^{k}\frac{\eta_{j}^2}{2} \prod_{i=j+1}^{k} \bigg( 1 -  \frac{\eta_i (2-\eta_i)}{2} \gamma\bigg).
\end{align*}
Then~\eqref{eq:des-l-EE-recursive} becomes
\begin{align}\label{eq:des-l-EE-D}
     \mathbb{E}[D_f^{x_{k+1}^*}(x_{k+1},\hat x)] \leq  D_k(\eta)\frac{\sigma^2}{\norm{\mathbf{A}}_{\square}^{2}}.
\end{align}
From~\eqref{eq:des-l-EE-D} we see that $D_k(\eta)$ can be written in a recurrence fashion as 
\begin{align}\label{eq:recursion-D}
   D_{k}(\eta) = \bigg( 1 -  \frac{\eta_k (2-\eta_k)}{2} \gamma\bigg) D_{k-1}(\eta) + \frac{\eta_k^2}{2}.
\end{align}
From~\eqref{eq:recursion-D} and~\eqref{eq:des-l-EE-D} we see that we would like to make $D_{k}(\eta)$ as small as possible by choosing appropriate values for the $\eta_{j}$.
From the definition of $D_k(\eta)$, it follows that $D_k$ only depends on $\eta_j$ for $j \leq k$ and hence, we set the partial derivative  $\frac{\partial}{\partial \eta_k} D_{k}(\eta)$ of $D_{k}(\eta)$ equal to zero and solve for $\eta_k$ and get 
\begin{align*}
  0 & = \frac{\partial}{\partial \eta_k} D_{k}(\eta)  \stackrel{~\eqref{eq:recursion-D}}{=} \frac{\partial}{\partial \eta_k} \bigg( \bigg( 1 -  \frac{\eta_k (2-\eta_k)}{2} \gamma\bigg) D_{k-1}(\eta) + \frac{\eta_k^2}{2}  \bigg)\\
  & = (-1+\eta_{k})\gamma D_{k-1}(\eta) + \eta_{k}
\end{align*} so that 
\begin{align}\label{eq:etak-betak}
    \eta_k = \frac{\gamma D_{k-1}(\eta)}{\gamma D_{k-1}(\eta) + 1} = \frac{\gamma \beta_k}{\gamma \beta_k + 1}
\end{align}
with $\beta_k \eqdef D_{k-1}(\eta)$ and $\beta_0 \eqdef \frac{\norm{\mathbf{A}}_{\square}^{2}\, D_f^{x_{0}^*}(x_{0},\hat x)}{\sigma^2}$. Moreover, using~\eqref{eq:etak-betak} we get
\begin{align}\label{eq:etak-betak-aux}
    -  \frac{\eta_k (2-\eta_k)}{2} \gamma \beta_k + \frac{\eta_k^2}{2} =  \frac{-\gamma\beta_k\eta_k}{2}.
\end{align}
Using~\eqref{eq:etak-betak-aux} in the recursion~\eqref{eq:recursion-D} expressed with $\beta_{k}$ we get 
\begin{align}\label{eq:recursion-betak}
    \beta_{k+1} = D_{k}(\eta)= \bigg( 1 -  \frac{\eta_k (2-\eta_k)}{2} \gamma\bigg) \beta_k + \frac{\eta_k^2}{2} = \beta_k(1 - \frac{\gamma \eta_k}{2}).
\end{align}
With this derivation we get from Lemma~\ref{lem:descent-D} by lower bounding the Bregman distance with the norm (recall~\eqref{eq:D}):
\begin{corollary}\label{cor:decay-norm}
  Assume that in the context of Lemma~\ref{lem:descent-D} we choose $\eta_{k}$ recursively by~\eqref{eq:etak-betak} and have $\beta_{k}$ defined by $\beta_{0} \eqdef \norm{\mathbf{A}}_{\square}^{2}\,D_{f}^{x_0^{*}}(x_0,\hat{x})/\sigma^{2}$ and the recursion~\eqref{eq:recursion-betak}. Then it holds that 
\begin{align*}
\EE{\norm{x_{k}-\hat{x}}_2^{2}} \leq  2\frac{\sigma^2}{\norm{\mathbf{A}}_{\square}^{2}}\beta_{k}.
\end{align*}
\end{corollary}
We see that the decay of the quantity $\beta_{k}$ governs the convergence speed of Algorithm~\ref{alg:BK} towards the noise-free solution $\hat{x}$ and hence, we would like to quantify its decay.

We have 
\begin{align*}
  \frac{\gamma\beta_{k+1}-\gamma\beta_k}{\gamma} & = \beta_{k+1}-\beta_{k} \stackrel{\eqref{eq:recursion-betak}}{=} - \frac{\gamma\beta_{k}}2 \eta_{k} \stackrel{\eqref{eq:etak-betak}}{=} -\frac{\gamma\beta_{k}}{2(\gamma\beta_{k}+1)}\gamma\beta_{k} = -\frac{(\gamma\beta_k)^{2}}{2(\gamma\beta_{k}+1)}.
\end{align*}
Setting $v_{k}\eqdef \gamma\beta_{k}$ we get the recursion 
\begin{align}\label{eq:recursion-v}
\frac{v_{k+1}-v_{k}}{\gamma} = -\frac{1}2 \frac{v_{k}^{2}}{v_{k}+1}.
\end{align}

\
\begin{remark}\label{rem:estimate-decay}
  To get the asymptotic decay of $v_{k}$ we make the quick observation that since $v_{k}$ is positive and decreasing, we also have $v_{k+1}-v_{k} \leq -\tfrac{\gamma}{2(v_{0}+1)}v_{k}^{2}$ and from this one can deduce that
  \begin{align}\label{eq:estimated-rec}
    v_{k} \leq \left(\tfrac1{v_{0}} + \tfrac{\gamma}{2(v_{0}+1)}k\right)^{-1} =  \mathcal{O}(k^{-1})
  \end{align}
  and especially we see that $\beta_{k} = \mathcal{O}(k^{-1})$ and by Corollary~\ref{cor:decay-norm} that $\norm{x_{k}-\hat{x}} = \mathcal{O}(k^{-1/2})$. 
\end{remark}
This quick estimate is not good for small $k$ and we would like to get a sharper bound in this regime. As observed in\cite[Section 2.5]{marshall2023optimal} the recursion~\eqref{eq:recursion-v} is a forward-Euler approximation to the ordinary differential equation $\dot{u} = - u^{2}/(2u+2)$ which has the solution $u(t) = 1/W(e^{ t/2+c})$ with the Lambert-W function $W$ and $c$ coming from the initial condition $u(0) = u_{0}$, namely $c = \tfrac1{u_{0}} - \ln(u_{0})$. Moreover, it can be shown that $u$ is convex and decreasing and hence, the forward-Euler method is a positive lower bound for the solution. In conclusion, this shows that 
\begin{align}\label{eq:lambert-W}
v_{k} \leq \frac{1}{W(e^{\tfrac\gamma2 k + c})},\quad c = \tfrac{1}{v_{0}} - \ln(v_{0}).
\end{align}
In the variable $\beta_{k}$ this means that 
\begin{align}\label{eq:lambert-W-beta}
\beta_{k} \leq \frac{1}{\gamma 
 W(e^{\tfrac\gamma2 k + c})},\quad c = \tfrac{1}{\gamma\beta_{0}} - \ln(\gamma\beta_{0}).
\end{align}

\begin{remark}
  The original recursion from~\eqref{eq:recursion-v} decays much faster than~\eqref{eq:estimated-rec} initially if $v_{0}$ is large (cf. Figure~\ref{fig:error_estimates}). However, the $v_{k}$ which solves the recursion~\eqref{eq:recursion-v} also decay like $\mathcal{O}(1/k)$ asymptotically. The upper estimate from~\eqref{eq:lambert-W} is usually very close to~\eqref{eq:recursion-v} in this case and they could not be distinguished in Figure~\ref{fig:error_estimates}.
\end{remark}
\begin{figure}[htb]
  \centering
  \begin{tikzpicture}
    \begin{axis}[
      xlabel=$k$, 
      grid, 
      xmin = 0,
      xmax = 5200,
      ymin = 0,
      ymax = 110,
      axis lines=middle, 
      axis line style={-latex}, 
      width=8cm, 
      height=5cm, 
      x label style={at={(axis description cs:1.0,0)},anchor=west},
      ]
      \addplot[line width=1.5pt,color=olive] table [x=k, y=v, col sep=comma] {error_estimates_v_w_u.dat};
      \addlegendentry{\scriptsize original recursion~\eqref{eq:recursion-v}} 
      \addplot[line width=1.5pt,color=teal] table [x=k, y=w, col sep=comma] {error_estimates_v_w_u.dat};
      \addlegendentry{\scriptsize estimated from \eqref{eq:estimated-rec}} 
    \end{axis}
  \end{tikzpicture}
  \caption{Comparison of the original decay and the estimate from Remark~\ref{rem:estimate-decay}.}
  \label{fig:error_estimates}
\end{figure}
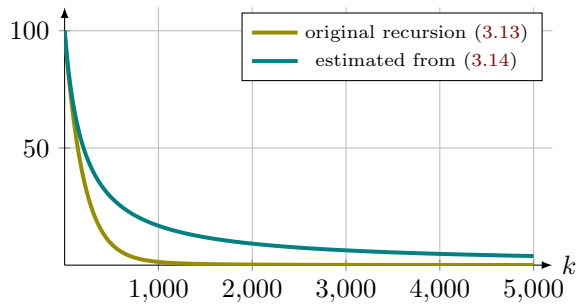
Next, we give some comments on the hyperparameters of our method.
\begin{remark}
   It holds $\eta_0 =  \frac{\gamma\beta_0}{\gamma\beta_0 +1} = \frac{\gamma \norm{\mathbf{A}}_{\square}^{2} D_f^{x_{0}^*}(x_{0},\hat x)}{\gamma \norm{\mathbf{A}}_{\square}^{2} D_f^{x_{0}^*}(x_{0},\hat x) + \sigma^2} \approx 1 $, since at the beginning of the method we have $\gamma \norm{\mathbf{A}}_{\square}^{2} D_f^{x_{0}^*}(x_{0},\hat x) \gg \sigma^2$. The method initially starts with a stepsize close to 1, so enjoys a linear rate up to some iteration and only decreases the learning rate once the error is smaller than the noise. Furthermore, $\eta_k \in [0, 1]$ is a non-increasing sequence.
\end{remark}
We introduce the function 
\begin{align}
  \label{eq:gk1}
  g(k) = \frac{\sigma^{2}}{\gamma W(ce^{\tfrac\gamma2 k})},\quad c = \frac{e^{\tfrac1{\gamma\beta_{0}}}}{\gamma\beta_{0}}
\end{align}
and note that from equation~\eqref{eq:lambert-W-beta} and Corollary~\ref{cor:decay-norm} we get that 
\begin{align*}
\EE{\norm{x_{k}-\hat{x}}_2^{2}} \leq \frac{2}{\norm{\mathbf{A}}_{\square}^{2}}g(k).
\end{align*}
The following corollaries provide more
intuition about the error bound function $g(k)$ and the two asymptotic regimes of the stepsize. First, we consider what happens when the noise vanishes.
\begin{lemma}
\begin{enumerate}
    \item (limit as $\sigma \to 0$). In the case where the variance of the noise $\sigma^2$ is small and the other parameters are fixed, we have
\begin{equation*}
    g(k) = \norm{\mathbf{A}}_{\square}^{2}e^{-\gamma k/2} \cdot D_f^{x_{0}^*}(x_{0},\hat x) + \mathcal{O}(\frac{\sigma^2}{ D_f^{x_{0}^*}(x_{0},\hat x)^2}) \quad \text{as} \quad \sigma \rightarrow 0
\end{equation*}
where the constant in the big-$\mathcal{O}$ notation depends on $\gamma$ and $k$. In particular, we have
\begin{equation*}
    \eta(k) \to 1 \quad \text{as} \quad \sigma \to 0.
\end{equation*}

\item (limit as $k \to \infty$). We have 
\begin{equation*}
  g(k) = \frac{2\sigma^2}{\gamma^2 k} \bigg(1 + \mathcal{O}(\frac{\text{ln}(k)}{k})\bigg), \quad \eta_k \approx \frac{2}{2 + \gamma k} \quad \text{as} \quad k \to + \infty
\end{equation*}
where the constant in the big-$\mathcal{O}$ notation depends on $\gamma$, $D_f^{x_{0}^*}(x_{0},\hat x)$ and $\sigma^2$. So that from Corollary~\ref{cor:decay-norm} we should expect
\begin{align*}
\EE{\norm{x_{k}-\hat{x}}_2} \leq \frac{2}{\gamma \norm{\mathbf{A}}_{\square}} \frac{\sigma}{\sqrt{k}}.
\end{align*}
\end{enumerate}
\end{lemma}

\begin{proof}
  The first order Taylor expansion of series of the Lambert-W function is $W(x) = x +  \mathcal{O}(x^2)$ (as $x \rightarrow 0$).
  Using $\beta_{0} = \norm{\mathbf{A}}_{\square}^{2} D_{f}^{x_{0}^{*}}(x_{0},\hat{x})/\sigma^{2}$ this gives
\begin{align*}
    g(k) &= \frac{\sigma^2}{\gamma} \frac{1}{\frac{\sigma^2}{\gamma \cdot  \norm{\mathbf{A}}_{\square}^{2}D_f^{x_{0}^*}(x_{0},\hat x)} e^{\gamma k/2} e^{\frac{\sigma^2}{\gamma \cdot \norm{\mathbf{A}}_{\square}^{2} D_f^{x_{0}^*}(x_{0},\hat x)}} + \mathcal{O}(\frac{\sigma^4}{ D_f^{x_{0}^*}(x_{0},\hat x)^2})} \\
    &=  \frac{\norm{\mathbf{A}}_{\square}^{2}D_f^{x_{0}^*}(x_{0},\hat x)e^{-\gamma k/2} }{ e^{\frac{\sigma^2}{\gamma \cdot  \norm{\mathbf{A}}_{\square}^{2}D_f^{x_{0}^*}(x_{0},\hat x)}} + \mathcal{O}(\frac{\sigma^4}{ D_f^{x_{0}^*}(x_{0},\hat x)^2})}
\end{align*}
Canceling terms and using the fact that $e^{\frac{\sigma^2}{\gamma \cdot D_f^{x_{0}^*}(x_{0},\hat x)} } = 1 + \mathcal{O}(\frac{\sigma^2}{ D_f^{x_{0}^*}(x_{0},\hat x)} )$ give
\begin{equation*}
    g(k) = \norm{\mathbf{A}}_{\square}^{2}e^{-\gamma k/2} D_f^{x_{0}^*}(x_{0},\hat x) + \mathcal{O}(\frac{\sigma^2}{ D_f^{x_{0}^*}(x_{0},\hat x)^2}) \quad \text{as} \quad \sigma \rightarrow 0
  \end{equation*}  
  Finally, from~\eqref{eq:etak-betak} we have
  \begin{align}\label{eq:cont-etak-betak}
    \eta_k = \frac{\gamma \beta_k}{\gamma \beta_k + 1} =  \frac{\gamma g(k)/\sigma^2}{\gamma g(k)/\sigma^2 + 1} = \frac{\gamma g(k)}{\gamma g(k) + \sigma^2}
\end{align}
and get $\eta(k) \to 1$ for $\sigma \to 0$.

In the standard Bregman-Kaczmarz algorithm with $\eta_{k}=1$ on obtains linear convergence with contraction factor $(1-\gamma/2)^{k}$. Here we get $e^{-\gamma k/2}$ which is just slightly worse.

Now we have a look at the case $k\to\infty$. Similarly to~\cite[Corollary 1.12]{marshall2023optimal} one can obtain that 
\begin{equation*}
  g(k) = \frac{2\sigma^2}{\gamma^2 k} \bigg(1 + \mathcal{O}(\frac{\text{ln}(k)}{k})\bigg) \quad \text{as} \quad k \to + \infty
\end{equation*}
and
\begin{equation}\label{eq:eta_limit}
  \eta(k) =  \frac{\gamma g(k)/\sigma^2}{\gamma g(k)/\sigma^2 + 1} \approx \frac{2}{2 + \gamma k} \approx \mathcal{O}(1/k) \quad \text{as} \quad k \to + \infty.
\end{equation}
\end{proof}

\section{Heuristic estimation of the hyperparameters}
\label{sec:heuristic}
In the previous section, we showed that Algorithm~\ref{alg:BK} with stepsize $\eta_{k}$ according to~\eqref{eq:etak-betak} produces iterates that indeed converge to the true noise-free solution $\hat{x}$ without ever using the clean right-hand side $b$, but only using noisy data. However, the algorithm needs to know explicitly the hyperparameters $\gamma$ and $\beta_{0} = \norm{\mathbf{A}}_{\square}^{2}\frac{D(x_0, \hat x)}{\sigma^2}$ in order to calculate the stepsizes. Setting the algorithm with incorrect parameters may result in a slower algorithm or can even destroy convergence towards the true solution.

In this part, inspired by~\cite{marshall2023optimal}, we introduce some heuristics estimations of those parameters in order to obtain good results in practice. We show how to estimate them from one additional run of the Bregman-Kaczmarz method with stepsize $\eta_k=1$. Let $x_0, x_1, \dots, x_N$ denote the iterates resulting from the additional run of the Bregman-Kaczmarz method with stepsize $\eta_k=1$ for $N$ iterations. We get from~\eqref{eq:des_l_EE} with $\eta_k=1$ :
\begin{align*}
    \mathbb{E}[D_f^{x_{k+1}^*}(x_{k+1},\hat x)] \leq \bigg( 1 -  \frac{\gamma}{2}\bigg) \mathbb{E}[D_f^{x_{k}^*}(x_{k},\hat x)] + \frac{\sigma^2}{2\norm{\mathbf{A}}_{\square}^{2}},
\end{align*}
and inductively we infer that
\begin{equation}
\label{eq:cvr}
    \mathbb{E}[D_f^{x_{k}^*}(x_{k},\hat x)] \leq  \bigg( 1 -  \frac{\gamma}{2}\bigg)^k \cdot D_f^{x_{0}^*}(x_{0},\hat x) + \frac{\sigma^2}{\gamma \norm{\mathbf{A}}_{\square}^{2}}.
\end{equation}
If we assume that the initial error is above the noise level and assume that eventually, the error stagnates because of the noise, then we expect that,
initially, $D_f^{x_{j+1}^*}(x_{j+1},\hat x) \approx (1 - \tfrac{\gamma}{2})D_f^{x_{j}^*}(x_{j},\hat x)$ and eventually $D_f^{x_{j}^*}(x_{j},\hat x) \approx \tfrac{\sigma^2}{\gamma \norm{\mathbf{A}}_{\square}^{2}}$ (see~\eqref{eq:cvr}). Using these estimates and the approximation $x_N \approx \hat x $ yields 
heuristic estimates for $\gamma$ and $\beta_0 = \tfrac{\norm{\mathbf{A}}_{\square}^{2}D(x_0, \hat x)}{\sigma^2}$ as follows:  We choose an index $N_{0}< N$ for which we assume that for $1\leq j\leq N_{0}$ it holds that $D_f^{x_{j+1}^*}(x_{j+1},\hat x) \approx (1 - \tfrac{\gamma}{2})D_f^{x_{j}^*}(x_{j},\hat x)$, i.e. we have $\tfrac\gamma2 \approx 1 - D_f^{x_{j+1}^*}(x_{j+1},x_{N})/D_f^{x_{j}^*}(x_{j},x_{N})$. We replace the quotient of the Bregman distances by their empirical mean over the $j$ and get the estimate
\begin{equation}\label{eq:estimate-gamma}
    \Tilde{\gamma} \eqdef 2 \left(1 - \frac{1}{N_0} \sum_{j=1}^{N_0} \frac{D_f^{x_{j}^*}(x_{j},x_N)}{D_f^{x_{j-1}^*}(x_{j-1},x_N)}\right).
\end{equation}
For $\beta_{0} = \tfrac{\norm{\mathbf{A}}_{\square}^{2} D_{f}^{x_0^{*}}(x_{0},\hat{x})}{\sigma^{2}}$ we choose an $N_{1}<N$ for which we assume that for $N_{1}\leq j<N$ it holds that $D_{f}^{x_{j}^{*}}(x_{j},\hat{x}) \approx \tfrac{\sigma^{2}}{\gamma \norm{\mathbf{A}}_{\square}^{2}}$, i.e. $\tfrac{\norm{\mathbf{A}}_{\square}^{2}}{\sigma^{2}} \approx \tfrac{1}{\gamma D_{f}^{x_j^{*}}(x_j,\hat{x})}$. We replace $\gamma$ by its estimate $\tilde\gamma$ and the Bregman distance by its empirical mean and get the estimate
\begin{align}\label{eq:estimate-beta0}
  \tilde\beta_{0} \eqdef \frac{D_{f}^{x_0^{*}}(x_{0},x_N)}{ \tfrac{\tilde \gamma}{N_{1}}\sum_{j=N-N_{1}}^{N-1} D_{f}^{x_j^{*}}(x_{j},x_{N})} = \left[\frac{\tilde \gamma}{N_1} \sum_{j=N-N_1}^{N-1} \frac{D_f^{x_{j}^*}(x_{j},x_N)}{D_f^{x_0^{*}}(x_0, x_N)} \right]^{-1}.
\end{align}

\section{Numerical experiments}
\label{sec:numerical_experiment}

We present several experiments to demonstrate the effectiveness of Algorithm \ref{alg:BK} under various conditions. The simulations were performed in \texttt{Python} on an Intel Core i7 computer with 16GB RAM. The code that produces the figures in this paper is available at \url{https://github.com/tondji/abk}.

\subsection{Synthetic experiments}
\label{sec:synthetic-experiments}

In this part, we want to illustrate the result of Corollary~\ref{cor:decay-norm}. For all experiments we consider $f(x) = \lambda \cdot \|x\|_1 + \frac{1}{2}\|x\|^2_2$, where $\lambda$ is the sparsity parameter which gives us $\nabla f^*(x) = S_{\lambda}(x)$. Note that $f$ is $1$-strongly convex but not smooth.
Synthetic data for the experiments are generated as follows:
all elements of the data matrix $\mathbf{A} \in \RR^{m\times n}$ are chosen independently and identically distributed from the standard normal distribution $\mathcal{N}(0, 1)$. To construct the sparse solution $\hat x \in \RR^n$, we generate a random vector with $s$ non-zero entries from the standard normal distribution and we set it as $\hat x$ and the corresponding right-hand side is $b = \mathbf{A}\hat x \in \RR^m$. We fixed the number of blocks $M$, all having the same size, and the individual noise variances to $\sigma_{i} = \sigma/\sqrt{M}$. 

We compared the five following methods:
\begin{enumerate}
\item \textbf{adaptive RSK}: The adaptive randomized sparse Kaczmarz method, i.e. Algorithm~\ref{alg:BK} with stepsize $\eta_{k}$ defined by~\eqref{eq:etak-betak}, $\lambda=0.05$, an exact value for $\beta_{0}$ (using the knowledge of the ground truth solution) and $\gamma=0.1$ from a grid search over a small range of values.
\item \textbf{RSK}: The randomized sparse Kaczmarz method, i.e. Algorithm~\ref{alg:BK} with $\eta_{k}=1$ (hence, neither $\gamma$ nor $\beta_0$ is needed, especially no ground truth knowledge is needed) and $\lambda=0.05$.
\item \textbf{heuristic adaptive RSK}: the heuristic adaptive randomized sparse Kaczmarz method, i.e. Algorithm~\ref{alg:BK} with $\eta_{k}$ from~\eqref{eq:etak-betak} and $\lambda=0.05$. The parameters $\gamma$ and $\beta_{0}$ we estimated from the run of RSK as described in Section~\ref{sec:heuristic}. We used $N_{0} = 400$ in the estimate~\eqref{eq:estimate-gamma} of $\gamma$ and $N_{1} = 100$ in the estimate~\eqref{eq:estimate-beta0} for $\beta_{0}$.  
\item \textbf{adaptive RK}: The adaptive randomized Kaczmarz method, i.e. Algorithm~\ref{alg:BK} with $\eta_{k}$ defined by~\eqref{eq:etak-betak}, $\lambda=0$, an exact value for $\beta_{0}$ (using the knowledge of the ground truth solution) and $\gamma=0.05$ from a grid search over small range values.
\item \textbf{RK}: The randomized Kaczmarz method, i.e. Algorithm~\ref{alg:BK} with $\eta_{k}=1$ (hence, neither $\gamma$ nor $\beta_0$ is needed, especially no ground truth knowledge is needed) and $\lambda=0$.
\end{enumerate}
The adaptive methods can be seen as an optimal baseline for the proposed approach: We determined the hyperparameter $\beta_{0}$ exactly and also optimized the performance of $\gamma$ by grid search. On the other hand, the heuristic adaptive RSK does not use any knowledge about the solution or the noise and only relies on the previous run of RSK (again, with no additional parameters). The only parameters that have to be chosen for the heuristic adaptive RSK are the constants $N_{0}$ and $N_{1}$. 

For each experiment, we run independent trials each starting with the initial iterate $x_0=0$ with the following parameters $m=2000, n=100, \sigma=0.05, s=10$ and $M=200$, i.e. each block is of size $m_i=10$. The best parameter $\gamma$ for every method aRSK and aRK were obtained by grid search over the set $\gamma \in \{0.005, 0.01, 0.05, 0.1, 1, 2\}$. We measured performance by plotting the average relative residual error $\| \mathbf{A}x_k - b\|_2/\|b\|_2$ and the relative error $\|x_k - \hat x \|_2/\|\hat x\|_2$ against the number of iterations in Figure~\ref{fig:example1}.

The RK and RSK methods will be our baseline comparisons. We did a small grid search and selected the $\lambda$ that gives the smallest relative error for RSK and aRSK. With $\lambda$ fixed for both RSK and aRSK, We used the same procedure to select $\gamma$ for the aRSK and aRK($\lambda=0$). Having obtained the parameter for RSK, we can estimate $\gamma$ and $\beta_0$  with which we run haRSK.

Figure~\ref{fig:example1} shows in a semilog plot of the relative residuals and the relative reconstruction errors for all methods. Note that the usual RK and RSK methods reduce the error exponentially at the beginning and stagnate when the order of the total noise level is reached as predicted by Eq~\eqref{eq:cvr}, whereas aRK, aRSK and haRSK thanks to their adaptive stepsize are able to bring down the residuals and the errors to zero according to Corollary~\ref{cor:decay-norm}. The parameters of aRK and aRSK giving the best results are obtained through a small grid search. After choosing $\lambda$ and $\gamma$ for aRK and aRSK we used the same $\lambda$ for the haRSK method and estimated $\gamma$ and $\beta_0$ without knowing the true solution. In Figure~\ref{fig:example1} we see that haRSK has a better performance than the other adaptive methods. The reason behind this can be due to the fact that the best parameters were not part of the grid search space.

Figure~\ref{fig:example1-eta-beta} shows in a semilog plot how the parameters $\eta_{k}$ and $\beta_{k}$ evolve during the iteration for all the adaptive methods. The stepsize for the haRSK has a slow decay compared to other stepsizes. All the stepsizes are starting near 1 and decreasing towards zero. The second plot shows the behavior of $\beta_k$ that governs the reconstruction error for different adaptive methods.

\begin{figure}[htb]%
    \centering
    \includegraphics[width=0.47 \textwidth]{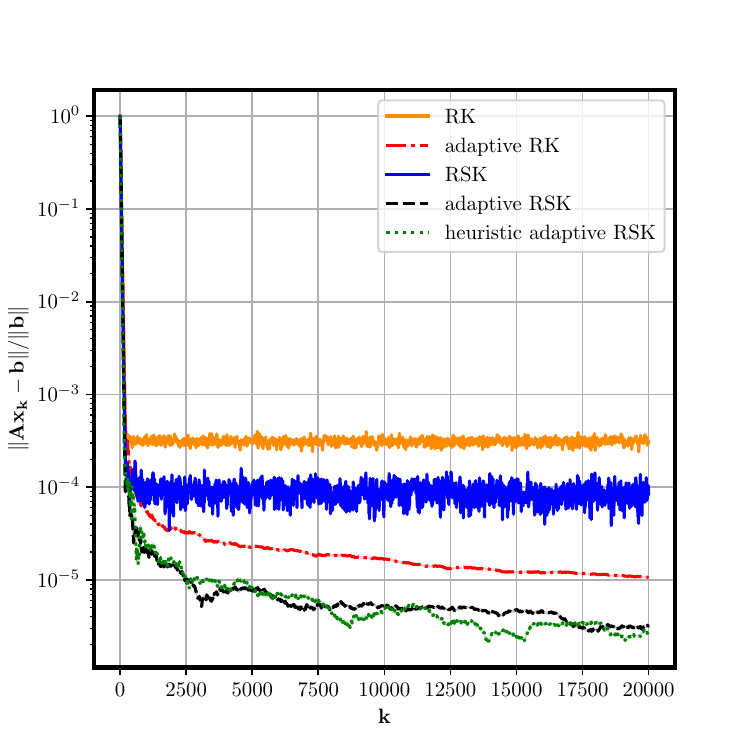}
    \qquad
    \includegraphics[width=0.47 \textwidth]{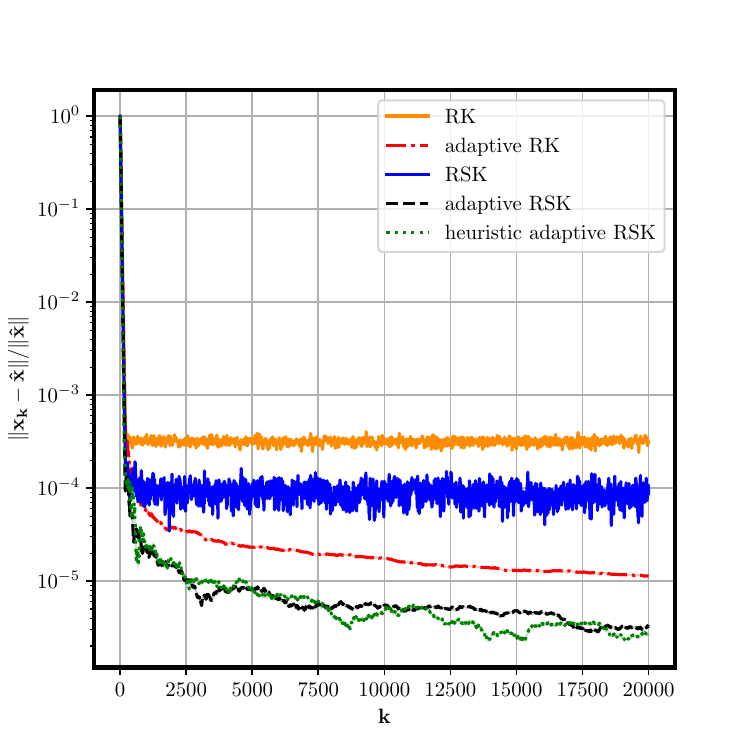}\vspace{-2ex}
    \caption{A comparison of the five methods as described in Section~\ref{sec:synthetic-experiments} for $m = 2000, n = 100,$ sparsity $s=10$, $\gamma_{aRK}=0.05, \gamma_{aRSK} = 0.1$, $\lambda=0.05$, $N_0=400, N_1=100$. From our heuristics we obtained $\Tilde{\gamma} = 0.0777$ and $\Tilde{\beta}_0 = 1022.60*10^{6}$. Left: Relative residual. Right: Relative error.}%
    \label{fig:example1}%
\end{figure}

\begin{figure}[htb]%
    \centering    
    \includegraphics[width=0.46 \textwidth]{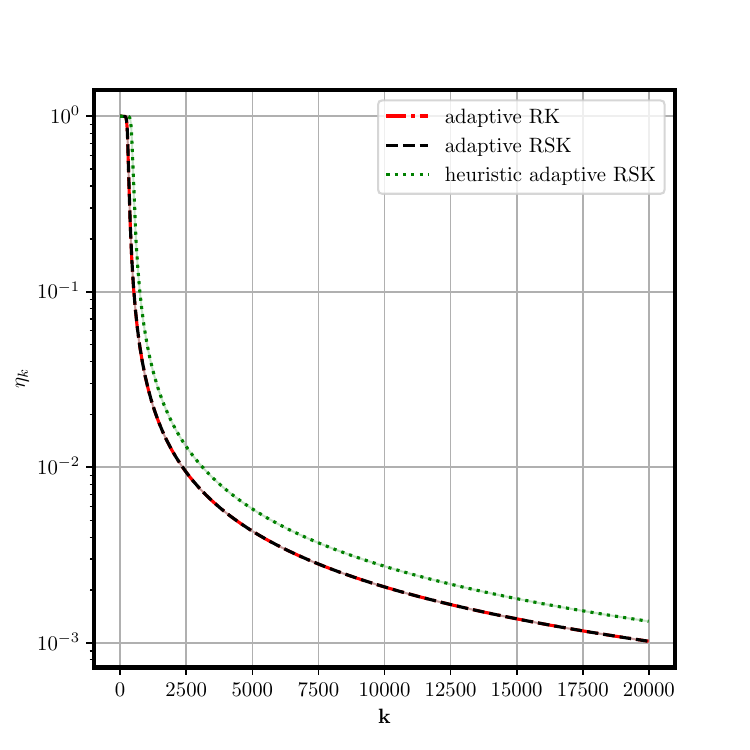}
    \qquad
    \includegraphics[width=0.46 \textwidth]{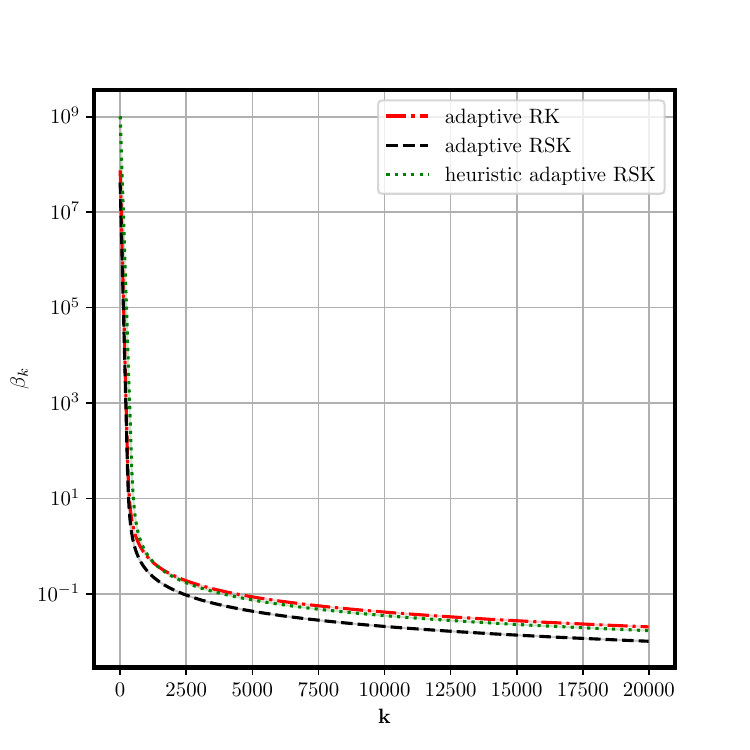}\vspace{-2ex}
  \caption{Evaluation of the stepsize $\eta_{k}$ (left) and the parameter $\beta_{k}$ (right) in the runs of the adaptive methods in Figure~\ref{fig:example1}.}
  \label{fig:example1-eta-beta}
\end{figure}

In another experiment, we investigated the influence of the number of blocks on convergence and recovery guarantee.
We generated an $m\times n$ matrix with $m=200$ rows and $n=100$ columns and a vector $\hat x\in\RR^{n}$ with $s=10$ non-zero entries. The right-hand side $b$ was always evaluated with independent noise with total noise level $\sigma$ such that the relative error was $10\%$. We investigated numbers of blocks $M\in \left\{ 200, 100, 20, 5, 1 \right\}$, each with equally sized blocks, i.e. block sizes of $1,2,10,40,200$. We chose $\lambda=1$. The hyperparameters $\beta_{0}$ and $\gamma$ depend on block size, and for a fair comparison we decided to use ``best possible parameters'', i.e. we always used the exact value of $\beta_{0}$ (using knowledge of ground truth $\hat x$) and determined the best $\gamma$ for each $M$ by a small grid search.
The results are shown in Figure~\ref{fig:example3} in a log-log plot. One sees that the residual and error go down well below the noise level and also decay further but very slowly in the end. Both error and residual go down fast in the beginning and do so faster for a larger number of blocks, i.e. for smaller blocks. This phenomenon could be observed for other settings as well but depends on an accurate choice of the parameter $\gamma$.

\begin{figure}[htb]%
    \centering
    \includegraphics[width=0.47 \textwidth]{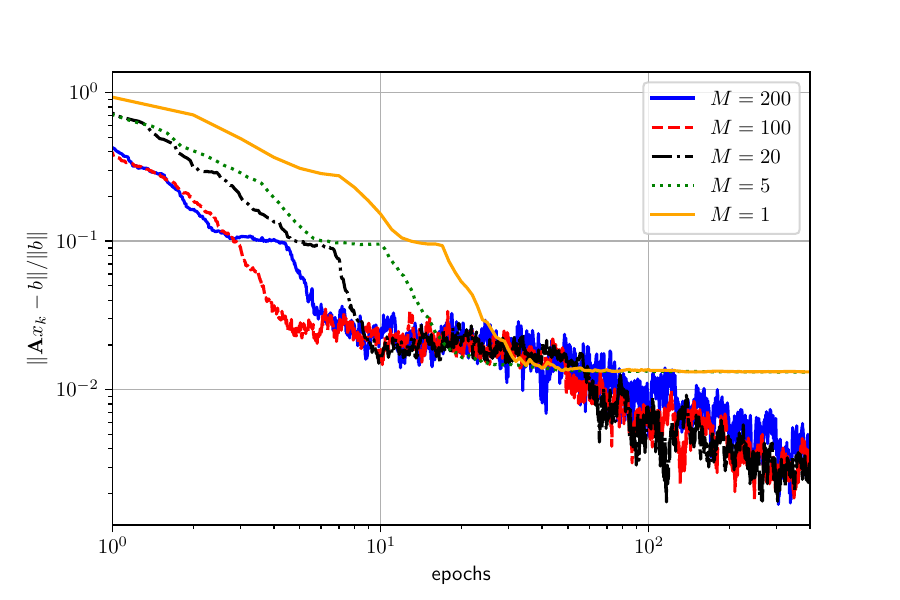}
    \qquad
    \includegraphics[width=0.47 \textwidth]{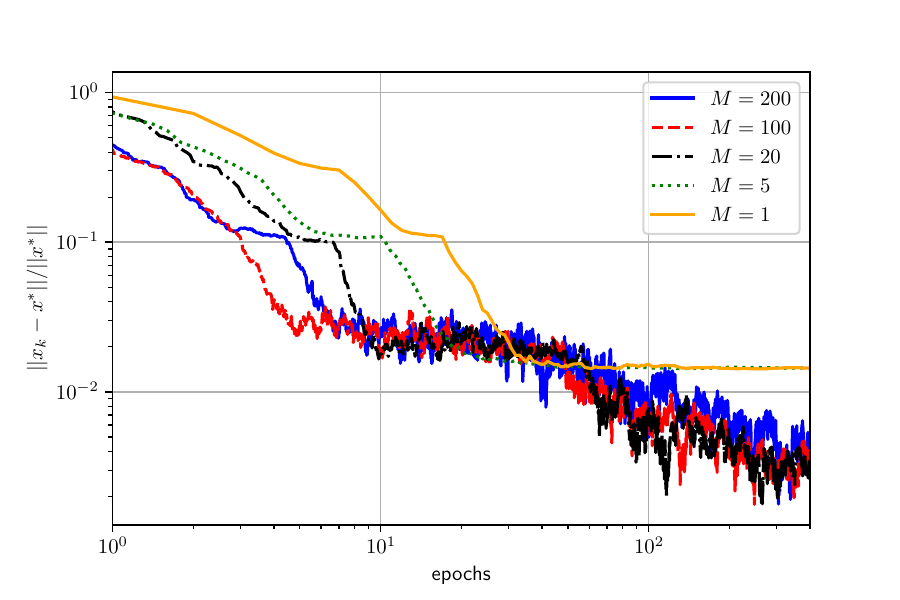}
    \caption{The effect of the block number $M$ on the relative error versus epochs for Algorithm \ref{alg:BK} (in the aRSK version) $m = 200, n = 100,$ sparsity $s=10$, $\gamma$ determined by grid search, $\beta_{0}$ exact. Left: Relative residual. Right: Relative error.}%
    \label{fig:example3}%
  \end{figure}

\subsection{Computerized tomography}
\label{sec:ct}

As an example on computerized tomography (CT), we used the implementation of the Radon transform from the \texttt{Python} library \texttt{skimage} and use it to build a system matrix for a parallel beam CT for a phantom of size $N\times N$ with $N=50$ and with $60$ equispaced angles. Hence, the system matrix $\mathbf{A}$ has size $3000\times 2500$, i.e. $m=3000$ and $n=2500$. We interpret the projection for each angle as one block $\mathbf{A}_{(i)}$, i.e. we have $M=60$ and each block has the size $m_{i}=50$. The ground truth solution $\hat x$ is fairly sparse (see below) and we generated the exact right-hand side simply as $b=\mathbf{A}x$. We used the total noise level of $10\%$, i.e. we choose $\sigma = 0.1\norm{b}$ and $\sigma_{i}=\sigma/\sqrt{M}$.

We used different methods for reconstruction, each run for $20$ epochs, i.e. for $60.000$ iterations and using the function $f(x)=\lambda\norm{x}_{1} + \tfrac12\norm{x}_2^2$:
\begin{enumerate}
\item \textbf{adaptive RSK}: The adaptive randomized sparse Kaczmarz method, i.e. Algorithm~\ref{alg:BK} with stepsize $\eta_{k}$ defined by~\eqref{eq:etak-betak}, $\lambda=30$, an educated guess for $\gamma$ and exact value for $\beta_{0}$ (using the knowledge of the ground truth solution).
\item \textbf{RSK}: The randomized sparse Kaczmarz method, i.e. Algorithm~\ref{alg:BK} with $\eta_{k}=1$ (hence, neither $\gamma$ nor $\beta_0$ is needed, especially no ground truth knowledge is needed) and $\lambda=30$.
\item \textbf{heuristic adaptive RSK}: the heuristic adaptive randomized sparse Kaczmarz method, i.e. Algorithm~\ref{alg:BK} with $\eta_{k}$ from~\eqref{eq:etak-betak} and $\lambda=30$. The parameters $\gamma$ and $\beta_{0}$ we estimated from the run of RSK as described in Section~\ref{sec:heuristic}. We used $N_{0} = 10.000$ in the estimate~\eqref{eq:estimate-gamma} of $\gamma$ and $N_{1} = 50.000$ in the estimate~\eqref{eq:estimate-beta0} for $\beta_{0}$. 
\item \textbf{adaptive RK}: The adaptive randomized Kaczmarz method, i.e. Algorithm~\ref{alg:BK} with $\eta_{k}$ defined by~\eqref{eq:etak-betak}, $\lambda=0$, an educated guess for $\gamma$ and exact value for $\beta_{0}$ (using the knowledge of the ground truth solution).
\end{enumerate}

In other words: The RSK method would be our baseline comparison. It can be applied by just choosing $\lambda$ by trial and error to adapt the sparsity of the outcome. With this, we can estimate the hyperparameter $\gamma$ and $\beta_{0}$ with which we run haRSK. This should produce a reconstruction which is much closer to the ground truth than RSK and even converge to the ground truth if the parameters were estimated exactly. The aRK and the aRSK method are just for further comparison.

All methods have been run for 20 epochs, i.e. each equation is used about 20 times and, consequently, each data point is queried about 20 times.

To show the relevance of the "independent noise" assumption, we conduct an additional experiment. Since the assumption of independent noise needs new measurements each time an equation is used, one could also try to do multiple measurements of the right hand side in advance and average the results to get improved measurement. To be precise, averaging $L$ independent measurement reduces the standard deviation of the noise by a factor of $1/\sqrt{L}$, i.e. we effectively reduce the noise level from 10\% to $\tfrac{10}{\sqrt{L}}\%$. 

\begin{enumerate}\setcounter{enumi}{4}
\item \textbf{RSK avg}: We sample $L$ noisy copies of $b$ with the same noise characteristic as other methods, we averaged those $L$ copies into one vector $b_{\text{avg}}$ and ran the standard RSK method with $b_{\text{avg}}$ and $\eta_k=1$ (hence, neither $\gamma$ nor $\beta_{0}$ is needed). We used $L=20$ to ensure that this methods uses the same number of measurement as the other methods. Moreover, we used $\lambda=30$ as in the other methods.
\end{enumerate}

In Figure~\ref{fig:ct-reconstructions} we show the ground truth data and the reconstructions by the adaptive RK, RSK, adaptive RSK and heuristic adaptive RSK. As can be seen, the sparse Kaczmarz methods all produce a clear background due to the sparsity enforcing $f$. The adaptive RSK and the heuristic adaptive RSK also lead to less noise in the non-zero regions. In order to evaluate the quality of the reconstructed image in Figure~\ref{fig:ct-reconstructions} ensuring that optimization processes result in outputs that are perceptually similar or close to the original data we report SSIM and PSNR, two commonly used metrics, in Table~\ref{tab:my-table}.

\begin{figure}[htb]
  \centering
  \includegraphics[width=\textwidth]{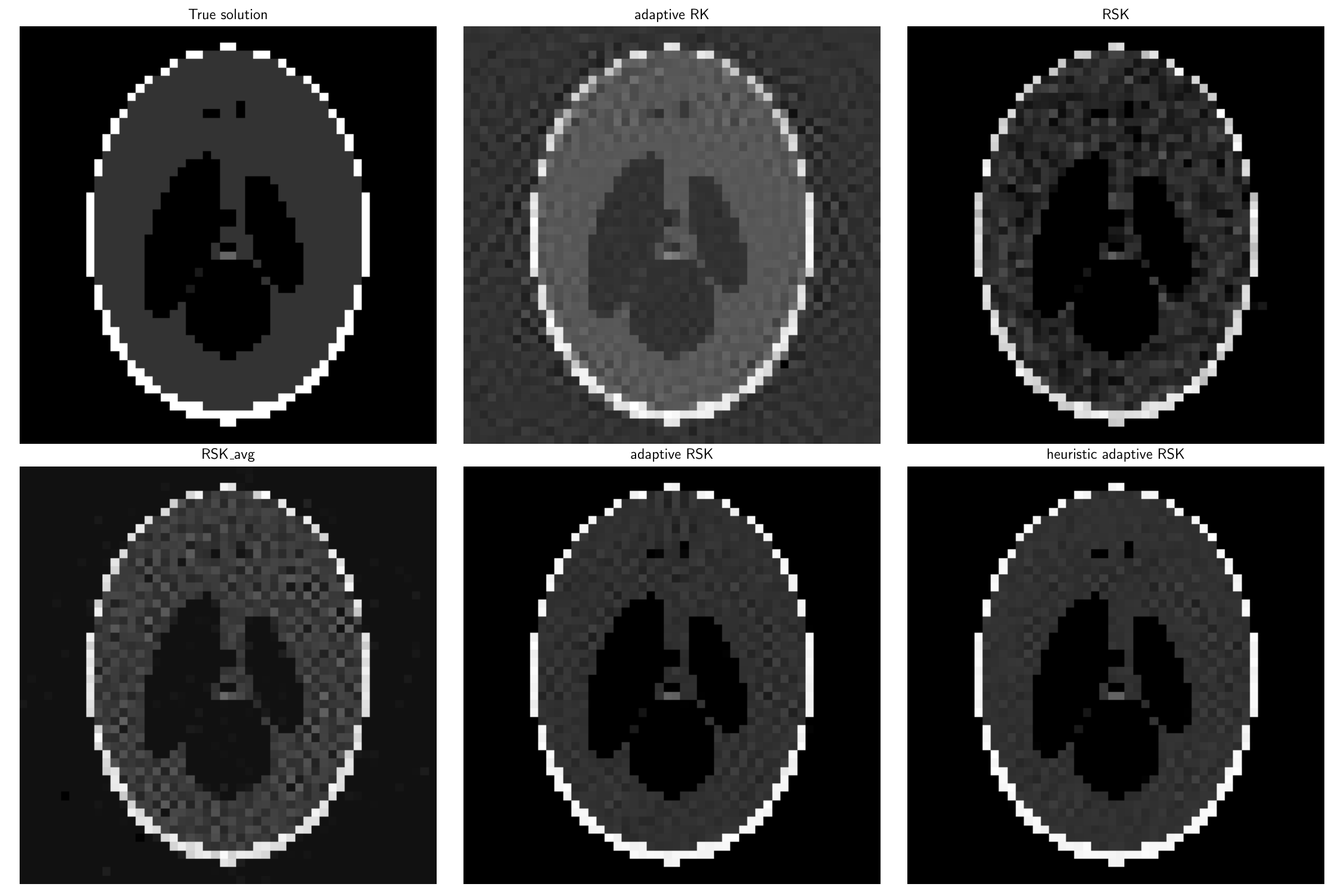}
  \caption{Ground truth solution and reconstruction by the different methods.}
  \label{fig:ct-reconstructions}
\end{figure}


\begin{table}[!ht]
\centering
\begin{tabular}{llllll}
\toprule
Methods & ARK & RSK & RSK avg & ARSK & hARSK \\\midrule
SSIM & 0.884 & 0.925 & 0.918 & \textbf{0.990} & \textbf{0.993} \\ 
PSNR & 27.302 & 28.240 & 28.119 & \textbf{38.97} & \textbf{40.407} \\ \bottomrule
\end{tabular}
\caption{Comparison of the SSIM and PSNR in Figure~\ref{fig:ct-reconstructions} for different methods: adaptive RK (ARK), RSK, RSK avg, adaptive RSK (ARSK) and heuristic adaptive RSK (hARSK).}
\label{tab:my-table}
\end{table}

SSIM~\cite{wang2004image} is used to measure how well the reconstructed output preserves the structural information of the original input. SSIM values range from $-1$ to $1$. A value of 1 indicates a perfect similarity, while a value of $-1$ indicates complete dissimilarity. PSNR is used to quantify the quality of the reconstructed image by comparing it to the original image. Higher PSNR values indicate better quality. As we can see from Table~\ref{tab:my-table}, our two methods ARSK and hARSK are the ones giving the best reconstruction of our CT problem.

Figure~\ref{fig:ct-residuals} shows the relative residuals and the relative reconstruction errors for all methods. It can be observed that $\text{RSK}_{\text{avg}}$ is able to decrease faster at the beginning due to low noise variance but both residual and the reconstruction quality settle for worse values than using the adaptive stepsize. While the RSK method brings the residual down to some value in the order of the total noise level, the adaptive methods are able to go way below that and it seems that they even go down further. This indicates that the choice of parameters is quite accurate. We stress that the heuristic adaptive RSK does not need any knowledge of the true solution. The only choice that the user has to make is $\lambda$ (to control the sparsity), and $N_{0}$ and $N_{1}$ for the estimation of $\gamma$ and $\beta_{0}$. We also see that the reconstruction error for adaptive RSK and heuristic adaptive RSK keeps decaying until the last epoch (and is expected to go down even further). We find it quite remarkable that this is possible even though the heuristic adaptive RSK method only sees corrupted data (with independent noise). As expected, the error in RSK and adaptive RK stagnates at some level (RSK because it is not adaptive and adaptive RK because it does not respect the sparsity). Table~\ref{tab:cpu_time} shows CPU times for different methods necessary to run for 20 epochs or to bring the relative error below a tolerance of $0.078$ for Figure~\ref{fig:ct-residuals}. We notice that ARSK and hARSK are the fastest methods to reach the tol meanwhile ARK and RSK are not able to reach the tolerance even after the total number of epochs.

\begin{figure}[htb]
  \centering
  \includegraphics[width=0.47\textwidth]{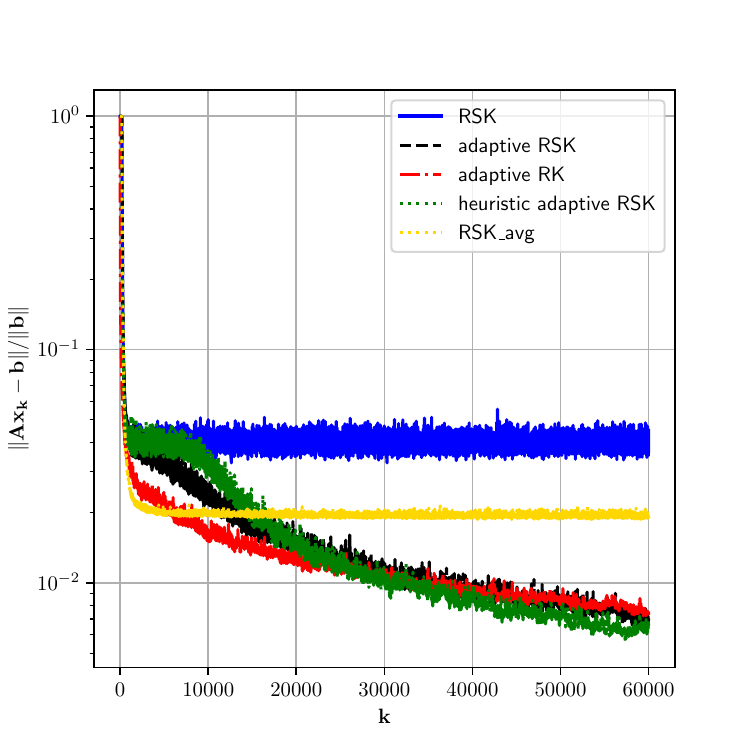}\qquad 
  \includegraphics[width=0.47\textwidth]{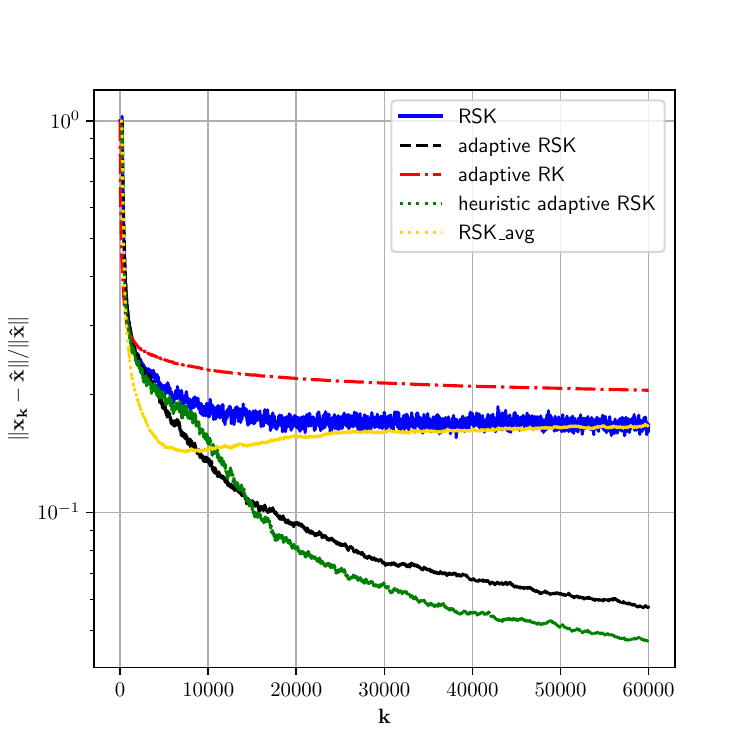}
  \caption{Left: Relative residuals for the CT reconstruction for the different methods. Right: Relative errors for the CT reconstruction.}
  \label{fig:ct-residuals}
\end{figure}

\begin{table}[htb]
\centering
\begin{tabular}{lr}
\toprule
 & CPU time (s) \\ \midrule
ARK             &  192.456 $^*$                         \\
RSK           &  264.749 $^*$                       \\
RSK avg           &  215.851 $^*$                       \\
ARSK           & \textbf{62.329}                         \\
hARSK          & \textbf{58.539}                         \\ \bottomrule
\end{tabular}
\caption{CPU time until a relative error falls below \texttt{tol}=0.078 for Figure~\ref{fig:ct-residuals}. Here $"*"$ indicates that the relative error was not achieved after the total number of iterations.}
\label{tab:cpu_time}
\end{table}

  To further investigate the influence of the number of measurements in Figure~\ref{fig:ct-residuals} for the RSK avg and adaptive RSK method, we conducted a further experiment: For different $L$ we generated $L$ measurements of the full right hand side. We run both methods for $L$ epochs, the rsk\_avg with the averaged measurements, and report the final relative residual and error for different values of $L$ in Table~\ref{tab:mult_measurements}. In this way, we used the same number of measurements and same number of iterations. We can see that the residuals of rsk\_avg get smaller for larger $L$ and are basically decaying like $1/\sqrt{L}$ as expected. The residuals of arsk decay similarly and always stay below the residuals of rsk\_avg. For the reconstruction error, the picture is not so clear, and the errors seem to stagnate. However, the reconstruction errors of arsk are consistently smaller than the ones for rsk\_avg.
  
\begin{table}[htb!]
  \centering
  \begin{tabular}{lrrrr}\toprule
    & \multicolumn{2}{c}{Residuals} & \multicolumn{2}{c}{Errors}\\
    $L$ & rsk\_avg & arsk & rsk\_avg & arsk\\\midrule
1  & 0.087  & 0.033\hphantom{8}  & 0.411  & 0.2\hphantom{70} \\
10 & 0.027  & 0.010\hphantom{8}  & 0.233  & 0.070 \\
20 & 0.019  & 0.0068 & 0.187  & 0.049 \\
30 & 0.014  & 0.006\hphantom{8}  & 0.135  & 0.047 \\
40 & 0.013  & 0.0049 & 0.131  & 0.041 \\
50 & 0.011  & 0.0041 & 0.129  & 0.034 \\
60 & 0.010  & 0.0039 & 0.116  & 0.031 \\
70 & 0.009  & 0.0038 & 0.103 & 0.041 \\
80 & 0.009  & 0.0032 & 0.097  & 0.034 \\   
90 & 0.008  & 0.0032 & 0.102 & 0.033 \\  
100& 0.007  & 0.0034 & 0.109  & 0.028 \\\bottomrule
  \end{tabular}
\caption{Final relative residuals and errors for both methods for different values of $L$ after $L$ epochs for the CT reconstruction. Here arsk stands for adaptive RSK and rsk\_avg stands for the standard randomized sparse Kaczmarz method where the $L$ measurements have been averaged.}
\label{tab:mult_measurements}
\end{table}

It follows from~\eqref{eq:cvr} that the iterates from Algorithm~\ref{alg:BK} with $\eta_k=1$ satisfy :
\begin{equation*}
    \mathbb{E}[\|x_{k} -\hat x\|^2_2] \leq  ( 1 -  q)^k \cdot 2f(\hat x) + \frac{\sigma^2}{q \|\mathbf{A}\|_{\square}^{2}}.
\end{equation*}
where $q = \frac{\gamma}{2}$ and $f(\hat x) = D_f^{x_{0}^*}(x_{0},\hat x)$. This means that the method will arrive in a ball of radius $\sigma/(\sqrt{q}\norm{A}_{\square})$ around $\hat{x}$. By picking $L$ noisy r.h.s and averaging them, we can reduce the radius of the ball of $\sigma/(\sqrt{Lq}\norm{A}_{\square})$.
For the proposed method, in the limit as $k \to \infty$, we have roughly
\begin{equation}
\label{eq:cvr3}
    \mathbb{E}[\|x_{k} -\hat x\|^2_2] \leq \frac{2}{ \|\mathbf{A}\|_{\square}^{2}} \cdot g(k),
\end{equation}
with $g(k) = \frac{\sigma^2}{2q^2k}$. So that the proposed method is better as soon as
\begin{align*}
    \frac{\sigma^2}{qL \|\mathbf{A}\|_{\square}^{2}} &\geq  \frac{2}{ \|\mathbf{A}\|_{\square}^{2}} \cdot \frac{\sigma^2}{2q^2k},
\end{align*}
i.e. for $k\geq L/q$.

We remark that the choice of $N_{0}$ and $N_{1}$ is crucial for a good estimate of the hyperparameters. The hyperparameter $\gamma$ is related to the asymptotic exponential decay of the Bregman distance. This parameter is quite difficult to estimate since the RSK method converges and reduces the Bregman distance to the solution much faster in the first iterations than in the asymptotic regime. On the other hand, the asymptotic regime is cluttered with effects from the noise. We observed that a fairly large $N_{0}$ leads to good results in most experiments. For the $\beta_{0}$ used an even larger number, i.e. we used $N_{1}$ large but only so large that the pre-asymptotic convergence has already happened.

\section{Conclusion}
\label{sec:conclusion}
In this work, we have proposed the block Bregman Kaczmarz using adaptive stepsize for solving finite dimensional linear inverse problems under the assumption of independent noise. In this setting, one never sees the noise-free right-hand side but always a new noisy version with noise being independent of every previous information and identically distributed. We showed that with a well-chosen stepsize~\eqref{eq:etak-betak} we are able to converge to the exact solution if we have access to a sufficient number of equations with independent noise. We gave a general error bound in terms of total noise variance $\sigma^2$, the block number $M$ and the constant $\theta(\hat x)$. The stepsize depends on two parameters: the signal-to-noise ratio $ D_{f}^{x_0^{*}}(x_{0},\hat{x})/\sigma^{2}$ and the rate parameter $\gamma$. However, we showed in Section~\ref{sec:heuristic} how to estimate these unknown parameters after having chosen $N_0, N_1$ and demonstrated in various scenarios in experiments the effectiveness of these estimations. We have shown in the experiments part that using a block size of 2 can already speed up the convergence a lot but in general taking a bigger block size gives us faster convergence results. It would be interesting to investigate how the actual rate $\mathcal{O}(k^{-1/2})$ to recover the exact solution can be accelerated.

The heuristic estimates of the hyperparameters $\gamma$ and $\beta_{0}$ depend on the choices of $N_{0}$ and $N_{1}$, respectively. While these values play a big role, one can usually get good values for these values by inspecting the decay of the residual of a standard Bregman-Kaczmarz run, keeping in mind which regimes should be used for the estimation of the hyperparameters (steady linear decay for $\gamma$ and stagnation for $\beta_{0}$).

In this paper, we assumed that the blocks are fixed for the full run of the algorithm. We think that the results can be extended to a framework where the blocks of rows are newly sampled in each iteration (cf.~\cite{necoara2019faster,tondji2022faster}) and expect that similar results than in this paper can be derived.

\bibliographystyle{plain}  
\bibliography{ref}   
\end{document}